\documentclass{article}

\usepackage[utf8]{inputenc}
\usepackage[T1]{fontenc}

\usepackage{bm}
\usepackage{dsfont}
\usepackage{csquotes}

\usepackage{xspace}
\usepackage{nicefrac}
\usepackage[super]{nth}
\usepackage{microtype}

\usepackage[english]{babel}

\usepackage{enumitem}

\usepackage{booktabs}
\usepackage{makecell}

\usepackage[backend=biber,
    style=authoryear-comp,
    natbib=true,%
    url=false,isbn=false,doi=false,%
    date=year,
    maxcitenames=2,
    maxbibnames=10]{biblatex}
\usepackage{graphicx}
\graphicspath{{./figures/}}

\usepackage{subcaption}

\usepackage{xcolor}

\usepackage{algorithm}
\usepackage[noend]{algpseudocode}
\algrenewcommand{\algorithmiccomment}[1]{$\vartriangleright$ #1}
\algrenewcommand{\algorithmicreturn}{\textbf{Return: }}
\algnewcommand\algorithmicinput{\textbf{Input: }}
\algnewcommand\Input{\State \algorithmicinput}

\usepackage{hyperref}
\newcommand\myshade{50}
\colorlet{mylinkcolor}{blue}
\colorlet{mycitecolor}{green}
\colorlet{myurlcolor}{red}
\hypersetup{
  linkcolor  = mylinkcolor!\myshade!black,
  citecolor  = mycitecolor!\myshade!black,
  urlcolor   = myurlcolor!\myshade!black,
  colorlinks = true,
  breaklinks = true,
}

\usepackage{url}

\makeatletter
\addto\extrasenglish{%
\renewcommand{\sectionautorefname}{\S\@gobble}%
\renewcommand{\subsectionautorefname}{\S\@gobble}%
\renewcommand{\subsubsectionautorefname}{\S\@gobble}%
}
\makeatother

\usepackage{changepage}                 
\newlength{\offsetpage}
\newenvironment{widepage}{\begin{adjustwidth}{-\offsetpage}{-\offsetpage}%
    \addtolength{\textwidth}{2\offsetpage}}%
{\end{adjustwidth}}

\usepackage{amsmath}
\usepackage{amsfonts} 
\usepackage{amssymb}
\usepackage{mathrsfs}
\usepackage{mathtools} 
\mathtoolsset{showonlyrefs}
\usepackage{empheq}

\usepackage{amsthm}

\newtheorem{proposition}{Proposition}

\newtheorem{lemma}{Lemma}

\newtheorem{assumption}{Assumption}

\newcommand{\NN}{\mathbb{N}}

\newcommand{\PP}{\mathbb{P}}

\newcommand{\RR}{\mathbb{R}}

\newcommand{\EE}{\mathbb{E}}

\newcommand{\Cc}{\mathcal{C}}

\newcommand{\Ff}{\mathcal{F}}

\newcommand{\Mm}{\mathcal{M}}
\newcommand{\Nn}{\mathcal{N}}
\newcommand{\Oo}{\mathcal{O}}
\newcommand{\Pp}{\mathcal{P}}

\newcommand{\Ww}{\mathcal{W}}
\newcommand{\Xx}{\mathcal{X}}
\newcommand{\Yy}{\mathcal{Y}}

\newcommand{\KL}{\textrm{KL}}

\newcommand{\Cz}{\bm{C}}

\newcommand{\K}{\bm{K}}

\renewcommand{\d}{\textrm{d}}

\newcommand{\f}{\bm{f}}
\newcommand{\g}{\bm{g}}

\renewcommand{\u}{\bm{u}}
\renewcommand{\v}{\bm{v}}

\newcommand{\al}{\alpha}

\newcommand{\be}{\beta}

\renewcommand{\phi}{\varphi}

\renewcommand{\leq}{\leqslant}
\renewcommand{\geq}{\geqslant}

\renewcommand{\epsilon}{\varepsilon}
\renewcommand{\imath}{\mathrm{i}}

\newcommand{\dotp}[2]{\langle #1,\,#2\rangle}

\newcommand{\norm}[1]{\left\| #1 \right\|}

\newcommand{\eqdef}{\triangleq}

\usepackage{accents}

\newcommand{\qandq}{ \quad \text{and} \quad }

\newlength{\restsubwidth}
\newlength{\restsubheight}
\newlength{\restsubmoreheight}
\newcommand{\rest}[2]{%
        \settowidth{\restsubwidth}{\ensuremath{#2}}
        \settoheight{\restsubheight}{\ensuremath{{}_{#2}}}
        \ensuremath{{#1\hskip 0.5pt}_{\vrule\kern2pt\parbox[b][%
        4pt][b]{\the\restsubwidth}{%
                        \ensuremath{{}_{#2}}}}}
        }

\newcommand{\enscond}[2]{ \left\{ #1 \;:\; #2 \right\} }
\newcommand{\Ctrans}[2]{T_{#2}(#1)}
\newcommand{\var}{\text{\upshape var}}

\def\icml{0}

\if\icml1
\input{icml_preamble.tex}
\else
\input{neutral_preamble.tex}
\fi

\begin{document}

\if\icml1
\input{icml_heading.tex}
\else
\maketitle
\fi

\begin{abstract}
  Optimal Transport (OT) distances are now routinely used as loss functions in ML tasks. Yet, computing OT distances between arbitrary (i.e. not necessarily discrete) probability distributions remains an open problem. This paper introduces a new online estimator of entropy-regularized OT distances between two such arbitrary distributions. It uses streams of samples from both distributions to iteratively enrich a non-parametric representation of the transportation plan. Compared to the classic Sinkhorn algorithm, our method leverages new samples at each iteration, which enables a consistent estimation of the true regularized OT distance. We provide a theoretical analysis of the convergence of the online Sinkhorn algorithm, showing a nearly-$\Oo(\frac{1}{n})$ asymptotic sample complexity for the iterate sequence. We validate our method on synthetic 1D to 10D data and on real 3D shape data.
\end{abstract}

\if\icml0
\medskip
\fi


Optimal transport (OT) distances are fundamental in statistical learning, both
as a tool for analyzing the convergence of various
algorithms~\citep{canas2012learning,dalalyan2019user}, and as a data-dependent
term for tasks as diverse as supervised learning~\citep{frogner2015learning},
unsupervised generative modeling~\citep{arjovsky2017wgan} or domain
adaptation~\citep{courty2016optimal}.
OT lifts a distance over data points living in a space $\Xx$ into a distance
on the space $\Pp(\Xx)$ of probability distributions over the space $\Xx$.
%
%
This distance has many favorable geometrical properties. In particular it allows one to compare distributions having disjoint supports. 
Computing OT distances is usually performed by sampling once from the input
distributions and solving a discrete linear program (LP), due
to~\citet{Kantorovich42}. This approach is numerically costly and statistically
inefficient \citep{weed2019sharp}. Furthermore, the optimisation problem depends on a fixed
sampling of points from the data. It is therefore not adapted to machine
learning settings where data is resampled continuously (e.g. in GANs), or
accessed in an online manner. In this paper, we develop an efficient
online method able to estimate OT distances between continuous distributions. It uses 
a stream of data to refine an approximate OT solution,
adapting the regularized OT approach to an online setting.

To alleviate both the computational and statistical burdens of OT, it is common
to regularize the Kantorovich LP.
The most successful approach in this direction is to use an entropic barrier penalty. 
When dealing with discrete distributions, this yields a problem that can be solved
numerically using Sinkhorn-Knopp's matrix balancing
algorithm~\citep{Sinkhorn64,sinkhorn1967concerning}.
This approach was pushed forward for ML applications by
\citet{cuturi2013sinkhorn}. Sinkhorn distances are smooth and amenable to GPU
computations, which make them suitable as a loss function in model training \citep{frogner2015learning, mensch_geometric_2019}.
The Sinkhorn algorithm operates in two distinct phases: draw samples from the
distributions and evaluate a pairwise distance matrix in the first phase;
balance this matrix using Sinkhorn-Knopp iterations in the second
phase.

This two-step approach does not estimate the true regularized OT
distance, and cannot handle samples provided as a stream, e.g. renewed at each
training iteration of an outer algorithm. A cheap fix is to use Sinkhorn over
mini-batches (see for instance~\citet{2018-Genevay-aistats} for an application
to generative modelling). Yet this introduces a strong estimation bias, especially in high
dimension ---see~\citet{fatras2019learning} for a mathematical analysis. In
contrast, we use streams of mini-batches to progressively enrich a consistent representation of the
transport plan.

\paragraph{Contributions.} Our paper proposes a new take on estimating optimal transport distances between continuous distributions. We make the following contributions:
\begin{itemize}[topsep=0pt, partopsep=0pt]
    \item We introduce an online variant of the Sinkhorn algorithm, that relies on
    streams of samples to enrich a non-parametric
    functional representation of the dual regularized OT solution.
    \item We establish the almost sure convergence of online Sinkhorn and derive asymptotic convergence rates 
    (\autoref{prop:convergence_true} and \ref{prop:rate}). We provide
    convergence results for variants.
    \item We demonstrate the performance of online Sinkhorn for estimating OT
    distances between continuous distributions and for accelerating the early
    phase of discrete Sinkhorn iterations.
\end{itemize}
\paragraph{Notations.} We denote $\Cc(\Xx)$ [$\Cc_+(\Xx)$] the set of [strictly
positive] continuous functions over a metric space $\Xx$, $\Mm^+(\Xx)$ and
$\Pp(\Xx)$ the set of positive and probability measures on $\Xx$, respectively.

\section{Related work}\label{sec:related}

\paragraph{Sinkhorn properties.} The Sinkhorn algorithm computes $\epsilon$-accurate
approximations of OT in $O(n^2/\epsilon^3)$ operations for $n$
samples~\citep{altschuler2017near} (in contrast with the $\Oo(n^3)$ complexity of exact OT \cite{goldberg1989finding}). Moreover, Sinkhorn distances suffer less from the
curse of dimensionality~\citep{2019-Genevay-aistats}, since the average error
using $n$ samples decays like $\Oo(\epsilon^{-d/2}/\sqrt{n})$ in dimension
$d$, in contrast with the slow $\Oo(1/n^{1/d})$ error decay of
OT~\citep{dudley_speed_1969,weed2019sharp}. Sinkhorn distances can further be sharpened
by entropic debiasing~\citep{2019-Feydy-aistats}. Our work is orthogonal, as we focus on estimating distances between continuous distributions.

\paragraph{Continuous optimal transport.} Extending OT computations to arbitrary
distributions (possibly having continuous densities) without relying on a fixed
a priori sampling is an emerging topic of interest. A special case is the
semi-discrete setting, where one of the two distributions is discrete. Without
regularization, over an Euclidean space, this can be solved efficiently using
the computation of Voronoi-like diagrams~\citep{merigot2011multiscale}. This
idea can be extended to entropic-regularized OT~\citep{cuturi2018semidual}, and
can also be coupled with stochastic optimization
methods~\citep{2016-genevay-nips} to tackle high=dimensional problems (see
 \cite{staib2017parallel} for an extension to Wasserstein barycenters). When dealing with arbitrary continuous densities, that are accessed through a
stream of random samples, the challenge is to approximate  the (continuous) dual
variables of the regularized Kantorovich LP using parametric or non-parametric
classes of functions. For application to generative model fitting, one can use
deep networks, which leads to an alternative formulation of Generative
Adversarial Networks (GANs)~\citep{arjovsky2017wgan} (see
also~\citet{seguy2018large} for an extension to the estimation of transportation
maps). There is however no theoretical guarantees for this type of dual
approximations, due to the non-convexity of the resulting optimization problem.
To our knowledge, the only mathematically rigorous algorithm represents potentials in reproducing
Hilbert space~\citep{2016-genevay-nips}. This
approach is generic and does not leverage the specific structure of the OT problem, so that in practice its convergence is slow. We show in Section~\autoref{sec:compare} that online Sinkhorn finds better potential estimates than SGD on RKHS representations.

\paragraph{Stochastic approximation (SA).} Our approach may be seen as SA \citep{robbins1951stochastic} for finding the roots of an
operator in a non-Hilbertian functional space. \cite{alber_stochastic_2012} studies
SA for solving fixed-points that are contractant in Hilbert spaces. Online Sinkhorn convergence relies on the contractivity of a certain operator in a non-Hilbertian metric, and requires a specific analysis. 
As both are SA instances, the online Sinkhorn algorithm resembles stochastic EM \citep{celeux_stochastic_1992}, though it cannot be interpreted as such.

\section{Background: optimal transport distances}

We first recall the definition of optimal transport distances between arbitrary distributions (i.e. not necessarily discrete), then review how these are estimated using a finite number of samples.

\subsection{Optimal transport distances and algorithms}

\paragraph{Wasserstein distances.} 

We consider a complete metric space $(\Xx,d)$ (assumed to be compact for simplicity), equipped
with a continuous cost function $(x,y) \in \Xx^2 \to C(x, y) \in \RR$ for any $(x,y) \in \Xx^2$ (assumed to be symmetric also for simplicity). 
Optimal transport lifts this \textit{ground cost} into a cost between probability
distributions over the space $\Xx$. 
The Wasserstein cost between two probability distributions $(\alpha, \beta) \in \Pp(\Xx)^2$ is defined as the minimal cost required to move each element of mass of $\alpha$ to each element of mass of $\beta$. It rewrites as the solution of a
linear problem (LP) over the set of transportation plans (which are probability distribution $\pi$ over $\Xx \times \Xx$):
\begin{equation}\label{eq:wass_0}
    \Ww_{C,0}(\alpha, \beta) \eqdef 
    \min_{\pi \in \Pp(\Xx^2)}
    \enscond{
    	\dotp{C}{\pi}
	}{ \pi_1=\al, \pi_2=\be},
\end{equation}
where we denote $\dotp{C}{\pi} \eqdef \int C(x,y) \d\pi(x,y)$. Here, $\pi_1 =
\int_{y\in \Xx} \d \pi(\cdot, y)$ and $\pi_2 = \int_{x\in \Xx} \d \pi(x, \cdot)$
are the first and second marginals of the transportation plan $\pi$. We refer to
\cite{santambrogio2015optimal} for a review on OT.
%

\paragraph{Entropic regularization and Sinkhorn algorithm.} 

The solutions of~\eqref{eq:wass} can be approximated by a strictly convex optimisation problem, where an entropic term is added to the linear objective to force strict convexity. The so-called Sinkhorn cost is then
\begin{equation}\label{eq:wass}
    \Ww_{C,\varepsilon}(\alpha, \beta) \eqdef 
    \min_{
    \substack{
        \pi \in \Pp(\Xx \times \Xx)
        \\\pi_1 = \alpha, \pi_2 = \beta
    }    
    } \dotp{C}{\pi} + \varepsilon \KL(\pi | \alpha \otimes \beta),
\end{equation}
where the Kulback-Leibler divergence is defined as $\KL(\pi | \alpha \otimes
\beta) \eqdef \int \log(\frac{\d\pi}{\d\al\d\be}) \d\pi$ (which is thus equal to
the mutual information of $\pi$).
$\Ww_{C,\varepsilon}$ approximates $\Ww_{C,0}(\alpha,\beta)$ up to an $\epsilon
\log(\epsilon)$ error \citep{2019-Genevay-aistats}. In the following, we set
$\varepsilon$ to $1$ without loss of generality, as $\Ww_{C, \varepsilon} =
\epsilon \Ww_{C / \varepsilon, 1}$, and simply write $\Ww$.
\eqref{eq:wass} admits a dual form, which is a maximization problem over the space of continuous functions:
\begin{equation}\label{eq:dual}
    F_{\alpha, \beta}(f, g) \eqdef \max_{(f, g) \in \Cc(\Xx)^2} \dotp{f}{\alpha} + 
    \dotp{g}{\beta}
    - \dotp{e^{f \oplus g - C}}{\alpha \otimes \beta} + 1, 
\end{equation}
where $\dotp{f}{\alpha} \eqdef \int f(x) \d\al(x)$ and $(f \oplus g - C)(x,y)
\eqdef f(x)+g(y)-C(x,y)$.
Problem \eqref{eq:dual} can be solved by closed-form alternated maximization, which corresponds to Sinkhorn's algorithm. At iteration $t$, the updates are simply
\begin{align}
    f_{t+1}(\cdot) &= \Ctrans{g_t}{\beta}, \quad
    g_{t+1}(\cdot) = \Ctrans{f_{t+1}}{\alpha},\\
    \Ctrans{h}{\mu} &\triangleq 
    - \log \int_{y \in \Xx}\!\! \exp(h(y) - C(\cdot, y))\d \mu(y).\label{eq:sinkhorn}
\end{align}
The operation $h \mapsto \Ctrans{h}{\mu}$  maps a continuous function to another
continuous function, and is a smooth approximation of the celebrated
$C$-transform of OT~\citep{santambrogio2015optimal}. We thus refer to it as a
\textit{soft C-transform}. Note that we consider \textit{simultaneous} updates
of $f_t$ and $g_t$ in this paper, as it simplifies our analysis.
The notation $f_t(\cdot)$ emphasizes the fact that $f_t$ and $g_t$ are
\textit{functions}. 
%

It can be shown that ${(f_t)}_t$ and ${(g_t)}_t$ converge in $(\Cc(\Xx),
\norm{\cdot}_{\text{var}})$ to a solution $(f^\star, g^\star)$ of
\eqref{eq:dual}, where $\norm{f}_{\text{var}} \eqdef \max_x f(x) - \min_x f(x)$
is the so-called variation norm. Functions endowed with this norm are only
considered up to an additive constant.  Global convergence is due to the strict
contraction of the operators $\Ctrans{\cdot}{\beta}$ and
$\Ctrans{\cdot}{\alpha}$ in the space $(\Cc(\Xx), \norm{\cdot}_{\text{var}})$
\citep{lemmens_nonlinear_2012}.

\subsection{Estimating OT distances with realizations}

When the input distributions are discrete (or equivalently when $\Xx$ is a
finite set), i.e. $\alpha = \frac{1}{n}\sum_{i=1}^n \delta_{x_i}$ and $\beta =
\frac{1}{n} \sum_{i=1}^n \delta_{y_i}$, the function $f_t$ and $g_t$ need only
to be evaluated on $(x_i)_t$ and $(y_i)_i$, which allows a proper
implementation. The iterations~\eqref{eq:sinkhorn} then correspond to the
\citet{sinkhorn1967concerning} algorithm over the inverse scaling vectors $\u_t
\triangleq {(e^{-f_t(x_i)})}_{i=1}^n, \v_t \triangleq
{(e^{-g_t(y_i)})}_{i=1}^n$:
\begin{equation}\label{eq-sinkhorn}
	\u_{t+1} = \K \frac{1}{n \v_t}
	\qandq
	\v_{t+1} = \K^\top \frac{1}{n \u_t}
\end{equation}
where $\K=(e^{-C(x_i,y_i)})_{i,j=1}^n \in \RR^{n \times n}$, and inversion is made pointwise. The Sinkhorn algorithm for OT thus
operates in two phases: first, the kernel matrix $\K$ is computed, with a cost in
$O(n^2 d)$, where $d$ is the dimension of $\Xx$; then each
iteration~\eqref{eq-sinkhorn} costs $O(n^2)$. The online Sinkhorn algorithm that we propose mixes
these two phases to accelerate convergence (see results in \autoref{sec:accelerating}).


\paragraph{Consistency and bias.}\label{sec:gradient}

The OT distance $\Ww_{C,0}(\alpha,\beta)$ and its regularized version
$\Ww_{C,\epsilon}(\alpha,\beta)$ can be approximated by the (computable)
distance between discrete realizations $\hat \alpha = \frac{1}{n} \sum_i
\delta_{x_i}$, $\hat \beta = \frac{1}{n} \sum_i \delta_{y_i}$, where ${(x_i)}_i$
and ${(y_i)}_i$ are i.i.d samples from $\alpha$ and $\beta$.  Consistency holds,
as $\Ww(\hat \alpha_n, \hat \beta_n) \to \Ww(\alpha,
\beta)$. Although this is a reassuring result, the sample complexity of
transport in high dimensions with low regularization remains high (see
\autoref{sec:related}).

The estimation of $\Ww(\alpha,\beta)$ may be improved using several i.i.d sets
of samples $(\hat \alpha_t)_t$ and ${(\hat \beta_t)}_t$. Those should be of
reasonable size to fit in memory and may for example come from a temporal
stream. \cite{2018-Genevay-aistats} use a Monte-Carlo estimate $\hat \Ww(\alpha,
\beta) = \frac{1}{T} \sum_{t=1}^T \Ww(\hat \alpha_t, \hat \beta_t)$. However,
this yields a biased estimation as the distance $\Ww(\alpha, \beta)$ and the
optimal potentials $f^\star=f^\star(\alpha, \beta)$ differ from their
expectation under sampling $\EE_{\hat \alpha \sim \alpha, \hat \beta \sim \beta}
[\Ww(\hat \alpha, \hat \beta)]$ and $\EE_{\hat \alpha \sim \alpha, \hat \beta
\sim \beta}[f^\star(\hat \alpha, \hat \beta)]$. In contrast, online Sinkhorn
consistently estimates the true potential functions (up to a constant) and the
Sinkhorn cost.


\section{OT distances from sample streams}

We now introduce an online adaptation of the Sinkhorn algorithm. We construct functional estimators of $f^\star$, $g^\star$
and $\Ww(\alpha,\beta)$ using successive discrete distributions of samples ${(\hat \alpha_t)}_t$
and ${(\hat \beta_t)}_t$, where $\hat\alpha_t \triangleq \frac{1}{n} \sum_{i=n_t +
1}^{n_{t+1}} \delta_{x_i}$, with  $n_0 \triangleq 0$ and $n_{t+1} \triangleq n_{t} +
n$. The size of the mini-batch $n$ may potentially depends on $t$.
${(\hat \alpha_t)}_t$ and ${(\hat \beta_t)}_t$ may be
seen as mini-batches of size $n$ within a training procedure.
%
%

\subsection{Online Sinkhorn iterations}
\label{sec-online-sink-iter}

The optimization trajectory ${(f_t, g_t)}_t$ of the continuous Sinkhorn
algorithm given by \eqref{eq:sinkhorn} is untractable as it cannot be
represented in memory. The exp-potentials $u_t \triangleq \exp(- f_t)$ and $v_t
\triangleq \exp(-g_t)$ are indeed infinitesimal mixtures of kernel functions
$\kappa_y(\cdot) \eqdef \exp(-C(\cdot, y))$ and $\kappa_x(\cdot) \eqdef \exp(-C(x, \cdot))$. 

We propose to construct finite-memory consistent estimates of $u_t$ and $v_t$ using
principles from stochastic approximation (SA) \cite{robbins1951stochastic}. We cast
the regularized OT problem as a root-finding problem of a function-valued operator $\Ff: \Cc_{+}(\Xx) \times \Cc_{+}(\Xx) \to \Cc(\Xx) \times \Cc(\Xx)$,
for which we can obtained unbiased estimates. Optimal potentials are indeed
exactly the roots of
\begin{equation}
    \Ff: (u, v) \to 
    \Big(u(\cdot) - \int_{y \in \Xx}
     \frac{1}{v(y)} \kappa_y(\cdot)  \d \beta(y),\quad
    v(\cdot) - \int_{x \in \Xx}
    \frac{1}{u(x)}  \kappa_x(\cdot)  \d \alpha(x)\Big).
\end{equation}
In particular, the simultaneous Sinkhorn updates rewrites as $(u_{t+1}, v_{t+1}) = (u_{t},
v_{t}) - \Ff(u_t, v_t)$ for all $t$. Importantly, $\Ff$ can be
evaluated without bias using two empirical measures $\hat \alpha$ and
$\hat \beta$, defining
\begin{equation}
    \hat \Ff_{\hat \alpha, \hat \beta}(u, v) \triangleq 
    \Big(u(\cdot) - \frac{1}{n} \sum_{i=1}^n
    \frac{1}{v(y_i)} \kappa_{y_i}(\cdot)\quad
    v(\cdot) - \frac{1}{n} \sum_{i=1}^n
   \frac{1}{u(x_i)}  \kappa_{x_i}(\cdot) \Big).
\end{equation}
By construction, $\EE_{\hat \alpha \sim \alpha, \hat \beta \sim \beta} [\hat \Ff_{\hat
\alpha, \hat \beta}] = \Ff$, and the images of $\hat \Ff$ admit a representation in memory.

\paragraph{Randomized Sinkhorn.}To make use of a stream of samples $(\hat \alpha_t, \hat
 \beta_t)_t$, we may simply replace $\Ff$ with $\hat \Ff$ in the Sinkhorn updates.
 This amounts to use noisy soft $C$-transforms in \eqref{eq:sinkhorn}, as we set
\begin{align}\label{eq:updates_naive}
    (u_{t+1}, v_{t+1}) &\triangleq (u_{t},
    v_{t}) - \hat \Ff_{\hat \alpha, \hat \beta}(u_t, v_t),\quad\text{i.e.}\\
     \hat f_{t+1} &= \Ctrans{\hat g_t}{\hat \beta_t},
    \qquad \hat g_{t+1} = \Ctrans{\hat f_{t+1}}{\hat \alpha_t}.
\end{align}
%
$\hat f_t$ and $\hat g_t$ are defined in memory by $(y_i, \hat g_{t-1}(y_i))_i$ and $(x_i,
\hat f_{t-1}(x_i))_i$. Yet the variance of the updates~\eqref{eq:updates_naive} does not
decay through time, hence this \textit{randomized Sinkhorn} algorithm does not
converge. However, we show in
\autoref{prop:markov} that the Markov chain ${(\hat f_t, \hat g_t)}_t$ converges
towards a stationary distribution that is independent of the potentials $\hat f_0$ and $\hat g_0$ used for
initialization.

\paragraph{Online Sinkhorn.}

To ensure convergence of $\hat f_t$, $\hat g_t$ towards some optimal pair of potentials
$(f^\star, g^\star)$, one must take more cautious steps, in particular past iterates should not be discarded. 
We introduce a learning
rate $\eta_t$ in Sinkhorn iterations, akin to the Robbins-Monro algorithm for finding
roots of vector-valued functions:
\begin{align}\label{eq:updates}
    (\hat u_{t+1}, \hat v_{t+1}) &\triangleq (1 - \eta_t) (\hat u_t, \hat v_t) 
    - \eta_t \hat \Ff_{\hat \alpha_t, \hat \beta_t}(\hat u_t, \hat v_t),\quad\text{i.e.}\\
    e^{-\hat f_{t+1}} &= (1 - \eta_t) e^{-\hat f_{t}} + 
    \eta_t e^{-\Ctrans{\hat g_t}{\hat \beta_t}}
\end{align}
Each update adds new kernel functions to a non-parametric estimation of $
u_t$ and $v_t$. The estimates $\hat u_t$ and $\hat v_t$ are defined by weights ${(p_{i,t},
q_{i,t})}_{i \leq n_t}$ and positions ${(x_i, y_i)}_{i \leq n_t} \subseteq~\Xx^2$:
\begin{align}\label{eq:param}
    e^{-\hat f_t(\cdot)} &= \hat u_t(\cdot) \triangleq \sum_{i=1}^{n_t} 
    \exp(q_{i,t} - C(\cdot, y_i)),\\
    e^{-\hat g_t(\cdot)} &= \hat v_t(\cdot) \triangleq \sum_{i=1}^{n_t}
    \exp(p_{i,t} - C(x_i, \cdot)).
\end{align}
The SA updates \eqref{eq:updates} yields simple vectorized updates for the weights
${(p_i,q_i)}_i$, leading to \autoref{alg:online_sinkhorn}. We perform
the updates for $q_i$ and $p_i$ in log-space, for numerical stability reasons.

\paragraph{Complexity.}

Each iteration of online Sinkhorn has complexity $\Oo(n_t\,n)$, due to the evaluation of the
distances $C(x_i, y_i)$ for all $(x_i)_{(0, n_t]}$ and $(y_i)_{(n_t, n_{t+1}]}$,
and the soft $C$-transforms in \eqref{eq:param}. Online Sinkhorn computes a
distance matrix $(C(x_i,y_j))_{i,j \leq n_t}$ on the fly, in parallel to
updating~$\hat f_t$ and~$\hat g_t$. In total, its computation cost after drawing
$n_t$ samples is $\Oo(n_t^2)$. Its memory cost is $\Oo(n_t)$; it increases with
iterations, which is a requirement for consistent estimation. Randomized
Sinkhorn with constant batch-sizes $n$ has a memory cost of $\Oo(n)$ and a
single-iteration computational cost of $\Oo(n^2)$.

\begin{algorithm}[t]
    \begin{algorithmic}
    \State \textbf{Input:} Dist. $\alpha$ and $\beta$, learning weights ${(\eta_t)}_t$, batch sizes ${(n(t))}_t$
    \textbf{Set} $p_i = q_i = 0$ for $i \in (0, n_1]$
    \For{$t= 0, \dots, {T-1}$}
        \State Sample $(x_i)_{(n_t, n_{t+1}]} \sim \alpha$, $(y_j)_{(n_t, n_{t+1}]} \sim \beta$.
            \State Evaluate $(\hat f_t(x_i))_{i=(n_t, n_{t+1}]}$,
             $(\hat g_t(y_i))_{i=(n_t, n_{t+1}]}$ using $(q_{i,t}, p_{i,t}, x_i, y_i)_{i=(0,n_{t}]}$ in \eqref{eq:param}.
             \State $q_{(n_t, n_{t+1}],t+1} {\gets} \log \frac{\eta_t}{n}
             + (\hat g_t(y_i))_{(n_t, n_{t+1}]}$,
             \qquad $p_{(n_t, n_{t+1}],t+1} {\gets} \log \frac{\eta_t}{n} 
             + (\hat f_t(x_i))_{(n_t, n_{t+1}]}$.
            \State $q_{(0, n_t],t+1} \gets q_{(0, n_t],t} + \log(1 - \eta_t)$, \qquad
            $p_{(0, n_t],t+1} \gets p_{(0, n_t],t} + \log(1 - \eta_t)$.
    \EndFor
    \State \textbf{Returns:} $\hat f_T : (q_{i,T}, y_i)_{(0, n_T]}$ and
    $\hat g_T : (p_{i,T}, x_i)_{(0, n_T]}$
    \end{algorithmic}
    \caption{Online Sinkhorn}\label{alg:online_sinkhorn}
\end{algorithm}

\subsection{Refinements}

\paragraph{Estimating Sinkhorn distance.} 

As we will see in \autoref{sec:analysis}, the iterations \eqref{eq:updates} only estimate potential functions up to a
constant. This is sufficient for minimizing a loss function involving a Sinkhorn
distance (e.g. for model training or barycenter estimation \citep{staib2017parallel}), as backpropagating through the Sinkhorn distance
relies only on the gradients of the potentials $\nabla_x f^\star(\cdot)$,
$\nabla_y g^\star(\cdot)$ \citep[e.g.][]{cuturi2018semidual}. With extra
$\Oo(n_t^2)$ operations, $(\hat f_t, \hat g_t)$ may be used to
estimate $\Ww(\alpha,\beta)$ through a final soft $C$-transform:
\begin{equation}\label{eq-dist-est}
    \hat \Ww_t \triangleq \frac{1}{2}\Big(\dotp{\bar \alpha_t}{f_t + 
    \Ctrans{\hat g_t}{\bar \alpha_t}}
     {+} \dotp{\bar \beta_t}{\hat g_t {+} \Ctrans{f_t}{\bar \alpha_t}}\Big),
\end{equation}
where $\bar \alpha_t \eqdef \frac{1}{n_{t}}\sum_{i=1}^{n_{t}} \delta_{x_i}$
and $\bar \beta_t$ are formed of all previously observed samples.

\paragraph{Fully-corrective scheme.} 

The potentials $\hat f_t$ and $\hat g_t$ may be improved by refitting the
weights $(p_i)_{(0, n_t]}$, $(q_j)_{(0, n_t]}$ based on all previously seen
samples.  For this, we update $\hat f_{t+1} = \Ctrans{g_t}{\bar \beta_t}$ and
$\hat g_{t+1} = \Ctrans{f_t}{\bar \alpha_t}$. This reweighted scheme (akin to
the fully-corrective Frank-Wolfe scheme from \cite{lacoste2015global}) has a
cost of $\Oo(n_t^2)$ per iteration. It requires to keep in memory (or recompute
on-the-fly) the whole distance matrix. Fully-corrective online Sinhorn  enjoys
similar convergence properties as regular online Sinkhorn, and permits the use
of non-increasing batch-sizes---see
\autoref{app:fully_corrective}. In practice, it can be used every $k$
iterations, with $k$ increasing with $t$. Combining partial and full updates can
accelerate the estimation of Sinkhorn distances (see \autoref{sec:accelerating}).



\paragraph{Finite samples.}Finally, we note that our algorithm
can handle both continuous or discrete distributions. When $\alpha$ and $\beta$
are discrete distributions of size $N$, we can store $p$ and $q$ as fixed-size
vectors of size~$N$, and update at each iterations a set of coordinates of size $n < N$. The resulting
algorithm is a \textit{subsampled} Sinkhorn algorithm for histograms, which is
detailed in \autoref{sec:sinkhorn_discrete}, \autoref{alg:discrete_online}. We show in \autoref{sec:exps} that it
is useful to accelerate the first phase of the Sinkhorn algorithm.


\section{Convergence analysis}\label{sec:analysis}

We show a
stationary distribution convergence property for the randomized Sinkhorn algorithm,
 an approximate convergence property for the online Sinkhorn algorithm with
fixed batch-size and an exact convergence result for online Sinkhorn with
increasing batch sizes, with asymptotic convergence rates. We make the following classical assumption on the cost regularity and compactness of $\alpha$ and $\beta$.

\begin{assumption}\label{ass:lip}
    The cost $C: \Xx \times \Xx \to \RR$ is $L$-Lipschitz, and $\Xx$ is  compact.
\end{assumption}

\subsection{Randomized Sinkhorn}

We first state a result concerning the randomized Sinkhorn algorithm~\eqref{eq:updates_naive}, proved in \autoref{sec:proof_markov}.

\begin{proposition}\label{prop:markov}
    Under \autoref{ass:lip}, the randomized Sinkhorn algorithm \eqref{eq:updates_naive} yields a time-homogeneous
    Markov chain ${(\hat f_t, \hat g_t)}_t$ which is $(\hat \alpha_s, \hat \beta_s)_{s \leq
    t}$ measurable, and converges in law towards a stationary distribution
    $(f_\infty, g_\infty) \in \Pp(\Cc(\Xx)^2)$ independent of the initialization
    point $(f_0, g_0)$.
\end{proposition}

This result follows from \citet{diaconis_iterated} convergence theorem on
iterated random functions which are contracting on average. We use the
fact that $\Ctrans{\cdot}{\hat \beta}$ and $\Ctrans{\cdot}{\hat \alpha}$ are
\textit{uniformly} contracting, independently of the distributions $\hat \alpha$ and
$\hat \beta$, for the variational norm $\Vert \cdot \Vert_{\text{var}}$.
Using the law of large number for Markov chains
\citep{breiman_strong_1960}, the (tractable) average $\frac{1}{t} \sum_{s=1}^t \exp(-\bar f_s)$
converges almost surely to $\EE[e^{-f_\infty}]
 \in \Cc(\Xx)$. This expectation verifies the functional equations
\begin{equation}
    \EE[e^{-f_\infty}] =
     \int_{y} \EE[e^{g_\infty(y) -C(\cdot, y)}]\d \beta(y) 
     \quad
    \EE[e^{-g_\infty}] =
    \int_{x} \EE[e^{f_\infty(x) -C(x, \cdot)}]\d \alpha(x) 
\end{equation}
These equations are close to the Sinkhorn fixed point equations, and
get closer as $\varepsilon$ increases, since $\varepsilon \EE[\exp(\pm f_\infty /
\varepsilon)] \to \EE[\pm f_\infty]$ as $\varepsilon \to \infty$. Running the random
Sinkhorn algorithm with averaging fails to provide exactly the dual solution, but solves an approximate problem.
%

\subsection{Online Sinkhorn}

We make the following \citet{robbins1951stochastic} assumption on the weight sequence. We then state an approximate convergence result for the online Sinkhorn algorithm with fixed batch-size $n(t) = n$.

\begin{assumption}\label{ass:weights}
    ${(\eta_t)}_t$ is such that
    $\sum \eta_t = \infty$ and $\sum \eta_t^2 < \infty$, $0 \leq \eta_t \leq 1$ for all $t > 0$.
\end{assumption}

\begin{proposition}\label{prop:convergence_approx}
    Under \autoref{ass:lip} and \ref{ass:weights}, the online Sinkhorn algorithm (\autoref{alg:online_sinkhorn}) yields a sequence $(f_t, g_t)$ that reaches a
    ball centered around $f^\star, g^\star$ for the variational norm $\Vert
    \cdot \Vert_{\var}$.
     Namely, there exists $T > 0$, $A > 0$ such that for all $t > T$, almost surely
    \begin{equation}
        \Vert f_t - f^\star \Vert_{\var}
        + \Vert g_t - g^\star \Vert_{\var} 
        \leq \frac{A}{\sqrt{n}} .
    \end{equation}
\end{proposition}

The proof is reported in \autoref{sec:proof_convergence_approx}. It is not possible
to ensure the convergence of online Sinkhorn with constant batch-size. This is a
fundamental difference with other SA algorithms, e.g. SGD on strongly convex
objectives (see \cite{moulines_non-asymptotic_2011}). This stems from the fact
that the metric for which $\text{Id} - \Ff$ is contracting is
not a Hilbert norm. The constant $A$ depends on $L$, the diameter of $\Xx$ and the regularity of potentials $f^\star$ and $g^\star$, but not on the dimension. It behaves like $\exp(\frac{1}{\varepsilon})$ when
$\varepsilon \to 0$.
Fortunately, we can show the almost sure convergence of the online Sinkhorn algorithm
with slightly increasing batch-size $n(t)$ (that may grow arbitrarily slowly for $\eta_t = \frac{1}{t}$), as specified in the following
assumption.

\begin{assumption}\label{ass:double_weights}
    For all $t > 0$, $n(t) = \frac{B}{w_t^2} \in \NN$ and $0 \leq \eta_t \leq
    1$. $\sum w_t \eta_t < \infty$ and $\sum \eta_t = \infty$.
\end{assumption}

\begin{proposition}\label{prop:convergence_true}
    Under \autoref{ass:lip} and
    \ref{ass:double_weights}, the online Sinkhorn algorithm converges almost surely:
    \begin{equation}
        \Vert \hat f_t - f^\star \Vert_{\var} + \Vert \hat g_t - g^\star \Vert_{\var} \to 0.
    \end{equation}
\end{proposition}

The proof is reported in \autoref{sec:proof_prop_convergence}. It relies on a uniform
law of large number for functions \citep[][chapter
19]{van_der_vaart_asymptotic_2000} and on the uniform contractivity of soft
$C$-transform operator \citep[e.g.][Proposition 19]{vialard2019elementary}. Consistency of the iterates is an original property---\cite{2016-genevay-nips} only show convergence of the OT value. Finally, using bounds from \cite{moulines_non-asymptotic_2011}, we derive  asymptotic rates of convergence for online Sinkhorn (see \autoref{app:proof_rate}), with
respect to the number of observed samples $N$. We write $\delta_N = \Vert \hat
f_{t(N)} - f^\star \Vert_{\var} + \Vert \hat g_{t(N)} - g^\star \Vert_{\var}$,
where $t(N)$ is the iteration number for which $n_t > N$ samples have been observed.

\begin{proposition}\label{prop:rate}
    For all $\iota \in (0, 1)$, $S > 0$ and $B \in \NN^\star$, setting $\eta_t =
    \frac{S}{t^{1 - \iota}}$, $n(t) = \lceil B t^{4\iota} \rceil$, there
    exists~$D > 0$ independant of $N$ and $N_0 > 0$ such that, for all $N >
    N_0$, $\delta_N \leq \frac{D}{N^{\frac{1 - \iota}{1 + 4 \iota}}}$.
\end{proposition}

Online Sinkhorn thus provides estimators of potentials whose asymptotic sample
complexity in variational norm is arbitrarily close to $\Oo(\frac{1}{N})$. To
the best of our knowledge, this is an original property. It also results in a
distance estimator $\hat \Ww_N$ whose complexity is arbitrarily close to
$\Oo(\frac{1}{\sqrt{N}})$, recovering existing asymptotic rates from
\cite{2019-Genevay-aistats}, for any Lipschitz cost. We derive non-asymptotic rates in \autoref{app:proof_rate}
(see \eqref{eq:non-asymptotic}), which make explicit the bias-variance trade-off
when choosing the step-sizes and batch-sizes. We also give the explicit form
of $D$; it does not depend on the dimension. For low $\epsilon$,
$D$ is proportional to $\exp(\frac{2}{\epsilon})$; the bound is therefore vacuous for
$\varepsilon \to 0$. Note that using growing
batch-sizes amounts to increase the budget of a single iteration over time: the
overall computational complexity after seeing $N$ samples is always $\Oo(N^2)$. 

\if\icml0
\paragraph{Batch-sizes and step-sizes.}

To provide practical guidance on choosing rates in batch-sizes $n(t)$ and
step-sizes $\eta_t$, we can parametrize $\eta_t = \frac{1}{t^a}$ and $n(t) = B
t^b$ and study what is implied by \autoref{ass:double_weights} and
\autoref{ass:total_growing}. We summarize the schedules for which convergence is
guarantees in \autoref{table:growing}. Note that in practice, it is useful to
replace $t$ by $(1 + r\, t)$ in these schedules. We set $r=0.1$ in all
experiments.
\begin{table}[t]
    \centering
    \caption{Schedules of batch-sizes and learning rates that ensures online Sinkhorn convergence.}

    \begin{tabular}{ccc}
        \toprule
        Param. schedule &Online Sinkhorn & Fully-corrective online Sinkhorn \\
        \midrule
        Batch size $\displaystyle n(t) = B t^b$ & $0 < b $ & $0 \leq b$ \\
        Step size $\displaystyle \eta_t = \frac{1}{t^a}$ &$\displaystyle a \geq 1 - \frac{b}{2}$ & 
        \parbox{5cm}{
        \begin{equation}
            \left\{
                \begin{aligned}
            a &> \frac{1}{2} - \frac{b}{2} \qandq b <1 \\
            a &\geq 0 \qandq b \geq 1
                \end{aligned}
                \right.
        \end{equation}}\\
        \bottomrule
    \end{tabular}
    \label{table:growing}
\end{table}

\paragraph{Mirror descent interpretation.} Online Sinkhorn can be interpreted as a non-convex stochastic mirror-descent, as detailed in \autoref{sec-mirror}. It provides an original interpretation of the Sinkhorn algorithm, different from recent work \citep{leger2019sinkhorn,mishchenko2019sinkhorn}.
\fi


\begin{figure}[t]
    \centering
    \begin{widepage}
    \includegraphics[width=\linewidth]{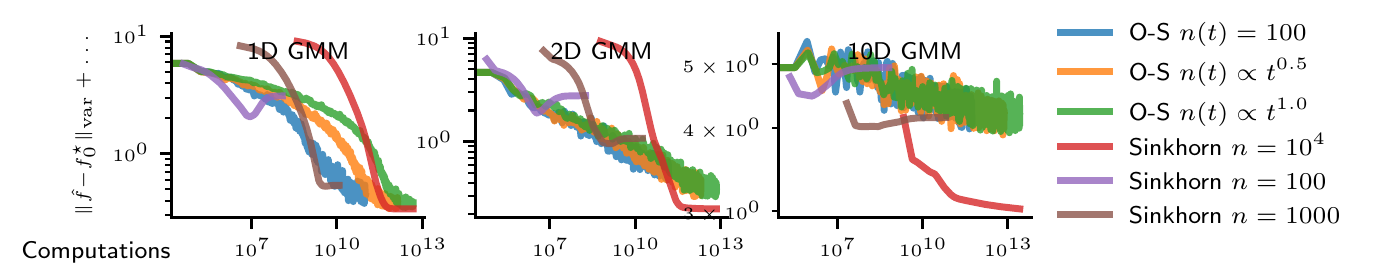}
    \end{widepage}
    \caption{Online Sinkhorn consistently estimate the true regularized OT potentials. Convergence here is measured in term of distance with potentials evaluated on a "test" grid of size $n=10^4$. Online-Sinkhorn can estimate potentials faster than sampling then scaling the cost matrix.}
    \label{fig:convergence}
\end{figure}

\begin{figure}[t]
    \centering
    \begin{widepage}
    \includegraphics[width=\linewidth]{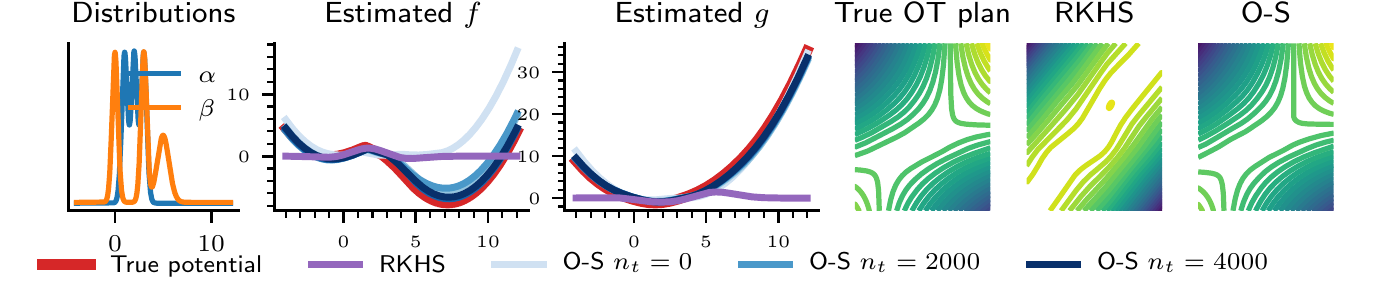}
    \end{widepage}
    \caption{Online Sinhkorn finds the correct potentials over all space, unlike SGD over a RKHS parametrization of the potentials. The plan is therefore correctly estimated everywhere.}
    \label{fig:potentials}
    \vspace{-1em}
\end{figure}

\section{Numerical experiments}\label{sec:exps}


The major purpose of online Sinkhorn (OS) is to handle OT between continuous
distributions.  We first show that it is a valid alternative to applying Sinkhorn
on a single realization of continuous distributions, using examples of Gaussian mixtures of varying dimensions.
We then illustrate that OS is able to estimate precisely
Kantorovich dual potentials, significantly improving the result obtained using SGD with RKHS
expansions~\citep{2016-genevay-nips}.
Finally, we show that OS is an efficient warmup strategy to accelerate Sinkhorn for discrete problems on several real and synthetic datasets.


\subsection{Continuous potential estimation with online Sinkhorn}\label{sec:continuous}

\paragraph{Data and quantitative evaluation.}\label{sec:online_exp}

We measure the performance of our algorithm in a
 continuous setting, where $\alpha$ and $\beta$ are  parametric
distributions (Gaussian mixtures in 1D, 2D and 10D, with 3, 3 and 5 modes, so
that $C_{\max} \sim 1$), from which we draw samples. In the absence of reference
potentials $(f^\star, g^\star)$ (which cannot be computed in closed form),
we compute ``test'' potentials $(f^\star_0, g^\star_0)$ on realizations $\hat
\alpha_0$ and $\hat \beta_0$ of size $10000$, using Sinkhorn. We then compare
OS to Sinkhorn runs of various size , trained on realizations $N=(100,1000,
10000)$ independent of the reference grid (to avoid reducing the problem to a
discrete problem between $\hat \alpha_0$ and $\hat \beta_0$). To measure
convergence, we compute $\delta_t = \Vert \hat f_t - f^\star_0 \Vert_{\var} +           
\Vert \hat g_t -  g_0^\star \Vert_{\var}$, evaluated on the grid defined by
$\hat \alpha_0$ and $\hat \beta_0$, which constitutes a Monte-Carlo approximation of the error.
We evaluate OS with and without full-correction, with
different batch-size schedules (see \autoref{app:online_exp}), as well as the randomized Sinkhorn algorithm. Quantitative results are average over 5 runs. We report quantitative results for $\epsilon = 10^{-2}$ and non fully-corrective online Sinkhorn in the main text, and all other curves in Supp.~\autoref{fig:convergence_all}. In
Supp.~\autoref{fig:gaussian}, we also report results for OT between Gaussians, which is a simpler and less realistic setup, but for which closed-form expressions of the potentials are known \cite{janati_entropic_2020}.

\paragraph{Comparison to SGD.}\label{sec:compare}
For qualitative illustration, on the 1D and 2D problem, we consider the main existing competing
approach \citep{2016-genevay-nips}, in which $f_t(\cdot)$ is parametrized as
$\sum_{i=1}^{n_t} \alpha_t \kappa(\cdot, x_i)$ (and similarly for $g_t$), where
$\kappa$ is a reproducing kernel (typically a Gaussian). This differs
significantly from online Sinkhorn, where we express $e^{-f_t}$ as a Gaussian
mixture. The dual problem \eqref{eq:sinkhorn} is solved using SGD, with convergence guarantees on the dual energy.  As advocated
by the authors, we run a grid search over the bandwidth parameter~$\sigma$ of
the Gaussian kernel to select the best performing runs.

\paragraph{Earlier potential convergence.}
We study convergence curves in \autoref{fig:convergence}, comparing algorithms
at equal number of multiplications. OS outperforms or matches Sinkhorn for $N=100$
and $N=1000$ on the three problems; it approximately matches the performance of
Sinkhorn on $N=10000$ new iterates on the 1D and 2D problems. On the two
low-dimensional problems, online Sinkhorn converges faster than Sinkhorn at the
beginning. Indeed, it initiates the computation of the potentials early, while the Sinkhorn
algorithm must wait for the cost matrix to be filled. This leads us
to study online Sinkhorn as a catalyser of Sinkhorn in the next paragraph. OS
convergence is slower (but is still noticeable) for the higher dimensional problem.
Fully-corrective OS performs better in this case (see Supp.~\autoref{fig:convergence_refit}). We also note that randomized Sinkhorn with batch-size $N$ performs on par with Sinkhorn of size $N$ (Supp.~\autoref{fig:convergence_randomized}).

\paragraph{Better-extrapolated potentials.} As illustrated in
\autoref{fig:potentials}, in 1D, online Sinkhorn refines the potentials $(\hat f_t, \hat g_t)_t$
until convergence toward $(f^\star, g^\star)$. Supp. \autoref{fig:potentials_2d} shows a visualisation for 2D GMM. As the parametrization~\eqref{eq:param} is
adapted to the dual problem, the algorithm quickly identifies the correct shape of
the optimal potentials---as predicted by \autoref{prop:convergence_true}. In
particular, OS estimates potentials with much less errors than SGD in a RKHS in
areas where the mass of $\alpha$ and $\beta$ is low. This allows to consistently estimate the transport plan, which cannot be achieved using SGD. SGD did not converge for $\epsilon < 10^{-1}$, while online Sinkhorn remains stable. OS does not require to set a bandwidth.

\setlength{\tabcolsep}{2pt}

\begin{figure}[t]
    \begin{widepage}
    \begin{minipage}{.7\linewidth}
    \includegraphics[width=\linewidth]{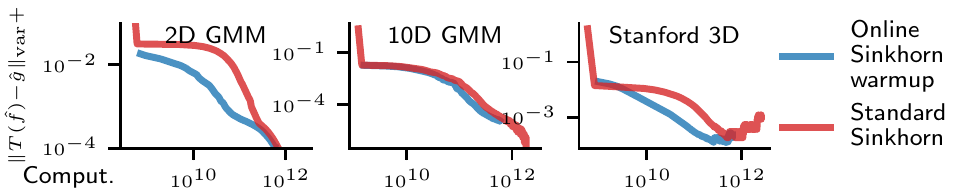}
    \end{minipage}%
    \hfill
    \begin{minipage}{.3\linewidth}
        \centering
        \small
        \begin{tabular}{lrrrr}
            \toprule
            $- \log \varepsilon$
            &  4 & 3 & 2 & 1 \\
            \midrule
            Stanford      &     5.3x &    4.4x &    2.3x &    1.2x \\
            10D GMM     &     1.4x &    1.5x &    1.3x &    1.2x \\
            2D GMM      &    17x &    3.7x &    1.3x &    2.0x \\
            \bottomrule
            \end{tabular}
    \end{minipage}
    \end{widepage}
    \stepcounter{table}
    \caption{Online Sinkhorn allows to warmup Sinkhorn during the evaluation of the cost matrix, and to speed discrete optimal transport. Table \thetable: Speed-ups provided by OS vs S to reach a $10^{-3}$ precision.}
    \label{fig:warmup}
    \vspace{-.8em}
\end{figure}

\subsection{Accelerating Sinkhorn with online Sinkhorn warmup}\label{sec:accelerating}

The discrete Sinkhorn algorithm requires to compute the full cost matrix $\Cz
\eqdef (C(x_i,y_i))_{i,j}$  of size $N \times N$, prior to estimating the
potentials $\f_1 \in \RR^N$ and $\g_1 \in \RR^N$ by a first $C$-transform. In
contrast, online Sinkhorn can progressively compute this matrix while computing
first sketches of the potentials. The extra cost of estimating the initial
potentials without full-correction is simply $2 N^2$, i.e. similar to filling-up
$\Cz$. We therefore assess the performance of \textit{online Sinkhorn as
Sinkhorn warmup} in a discrete setting. Online Sinkhorn is run with batch-size
$n$ during the first iterations, until observing each sample of $[1,N]$, i.e.
until the cost matrix $\Cz$ is completely evaluated. From then, the subsequent
potentials are obtained using full Sinkhorn updates. We consider the GMMs of
\autoref{sec:continuous}, as well as a 3D dragon from Stanford 3D scans
\cite{turk1994zippered} and a sphere of size $N=12000$. We measure convergence
using the error $\Vert \Ctrans{\hat f_t}{\alpha} - \hat g_t
\Vert_{\var} + \Vert \Ctrans{\hat g_t}{\beta} - \hat f_t \Vert_{\var}$,
evaluated on the support of $\alpha$ and $\beta$; this error goes to $0$. We use
$n(t) = \frac{N}{100} (1+0.1t)^{1/2}$---results vary little with the exponent.

\paragraph{Results.} We report convergence curves for $\varepsilon = 10^{-3}$ in
\autoref{fig:warmup}, and speed-ups due to OS in Table 1. Convergence curves for
different $\varepsilon$ are reported in Supp.~\autoref{fig:warmup_full}. The
proposed scheme provides an improvement upon the standard Sinkhorn algorithm.
After $N^2 d$ computations (the cost of estimating the full matrix $C$), both
the function value and distance to optimum are lower using OS: the full Sinkhorn
updates then relay the online updates, using an accurate initialization of the
potentials. The \textit{OS warmed-up} Sinkhorn algorithm then maintains its
advantage over the standard Sinkhorn algorithm during the remaining iterations.
The speed gain increases as $\varepsilon$ reduces and the OT problem becomes
more challenging. Sampling without replacement brings an additional speed-up.


\section{Conclusion}

We have extended the classical Sinkhorn algorithm to cope with streaming samples. The resulting online algorithm computes a non-parametric expansion of the inverse scaling variables using kernel functions. In contrast with previous attempts to compute OT between continuous densities, these kernel expansions fit perfectly the structure of the entropic regularization, which is key to the practical efficiently.
We have drawn links between regularized OT and stochastic approximation. This opens promising avenues to study convergence rates of continuous variants of Sinkhorn's iterations. Future work will refine the complexity constants and design adaptive non-parametric potential estimations.

\if\icml0
\section{Acknowledgements}
This work was supported by the European Research Council (ERC project NORIA). A.M thanks Anna Korba for helpful discussions on mirror descent algorithms, and Thibault Séjourné for proof-reading and relevant references.
\fi

\if\icml1
\section*{Broader impact}

This work is mostly a theoretical contribution on optimisation for comparing probability distributions. It has therefore no immediate societal impact to be expected.
\fi

\vfill

\printbibliography

\appendix

\onecolumn

\section{Proofs}\label{sec:proofs}

We first introduce two useful known lemmas, and prove the propositions in their order of appearance.

\subsection{Useful lemmas}

First, under \autoref{ass:lip}, we note that the soft $C$-transforms are
 uniformly contracting on the distribution space $\Pp(\Xx)$. This is clarified
 in the following lemma, extracted from \citet{vialard2019elementary},
 Proposition 19. We refer the reader to the original references for proofs.

\begin{lemma}\label{lemma:contractivity}
    Unser \autoref{ass:lip}, let $\kappa = 1 - \exp(-L
    \textnormal{diam}(\Xx))$. For all $\hat \alpha \in \Pp(\Xx)$ and $\hat \beta \in
    \Pp(\Xx)$, for all $f, f', g, g' \in \Cc(\Xx)$,
    \begin{equation}
        {\Vert \Ctrans{f'}{\hat \alpha} - 
        \Ctrans{f'}{\hat \alpha} \Vert}_{\var} \leq \kappa {\Vert f - f' \Vert}_{\var},
        \quad
        {\Vert \Ctrans{g}{\hat \beta} - 
        \Ctrans{g'}{\hat \beta} \Vert}_{\var} \leq \kappa {\Vert g - g' \Vert}_{\var}.
    \end{equation}
\end{lemma}

We will also need a uniform law of large numbers for functions. The following lemma is a consequence of Example 19.7 and
Lemma 19.36 of \citet{van_der_vaart_asymptotic_2000}, and is copied in Lemma B.6 in \citet{mairal_stochastic_2013}.

\begin{lemma}\label{lemma:lln}
    Under \autoref{ass:lip}, let $(f_t)_t$ be an i.i.d sequence in $\Cc(\Xx)$,
    such that $\EE[f_0] = f \in \Cc(\Xx)$. Then there exists $A > 0$ such that, for all $n > 0$,
    \begin{equation}
        \EE \sup_{x \in \Xx} | \frac{1}{n} \sum_{i=1}^n f_i(x) - f(x) |
        \leq \frac{A}{\sqrt{n}}.
    \end{equation}
\end{lemma}
Finally, we need a result on running averages using the sequence ${(\eta_t)}_t$. The following result stems from a simple Abel transform of the law of large number, and is established by \citet{mairal_stochastic_2013}, Lemma B.7.

\begin{lemma}\label{lemma:running}
    Let $(\eta_t)_t$ be a sequence of weights meeting \autoref{ass:weights}. Let
    $(X_t)_t$ be an i.i.d sequence of real-valued random variables with existing
    first moment $\EE[X_0]$. We consider the sequence ${(\bar X_t)}_t$ defined
    by $\bar X_0 \triangleq X_0$ and 
    \begin{equation}
        \bar X_t \triangleq (1 - \eta_t) \bar X_{t-1} + \eta_t X_t.
    \end{equation}
    Then $\bar X_t \to_{t \to \infty} \EE[X_0]$.
\end{lemma}

\subsection{Proof of \autoref{prop:markov}}\label{sec:proof_markov}

\begin{proof}
    We use Theorem 1 from \citet{diaconis_iterated}. For this, we simply note
    that the space $\Cc(\Xx) \times \Cc(\Xx)$ in which the chain ${x_t
    \triangleq (f_t, g_t)}_t$, endowed with the metric $\rho((f_1, g_1), (f_2,
    g_2)) = \Vert f_1 - f_2 \Vert_{\var} + \Vert g_1 - g_2 \Vert_{\var}$, is
    complete and separable (the countable set of polynomial functions are dense in this space, for example).
    We consider the operator $A_{\theta} \triangleq \Ctrans{\Ctrans{\cdot}{\hat
    \alpha}}{\hat \beta}$. $\theta \triangleq (\hat \alpha, \hat \beta)$ denotes
    the random variable that is sampled at each iteration. We have the following
    recursion:
    \begin{equation}
        x_{t+2} = A_{\theta_t}(x_t).
    \end{equation}
    
    From \autoref{lemma:contractivity}, for all $\hat \alpha \in \Pp(\Xx)$, $\hat \beta \in \Pp(\Xx)$, $A_{\theta}$
    with $\theta = (\hat \alpha, \hat \beta)$ is contracting, with module
    $\kappa_\theta < \kappa < 1$. Therefore
    \begin{equation}
        \int_{\theta} \kappa_\theta \d \mu(\theta) < 1, \qquad \int_{\theta}
         \log \kappa_\theta \d \mu(\theta) < 0.
    \end{equation}
    Finally, we note, for all $f \in \Cc(\Xx)$
    \begin{equation}
        \Vert \Ctrans{\Ctrans{f}{\hat \alpha}}{\beta} \Vert_{\infty} 
        \leq \Vert f \Vert_\infty + 2 \max_{x,y \in \Xx} C(x, y),
    \end{equation}
    therefore $\rho(A_\theta(x_0), x_0) \leq 2 \Vert x_0 \Vert_\infty + 2
    \max_{x,y \in \Xx} C(x, y)$ for all $\theta \ (\hat \alpha, \hat \beta)$.
    The regularity condition of the theorem are therefore met. Each of the
    induced Markov chains ${(f_{2t}, g_{2t})}_t$ and ${(f_{2t + 1}, g_{2t +
    1})}_t$ has a unique stationary distribution. These stationary distributions
    are the same: the stationary distribution is independent of the
    initialisation and both sequences differs only by their initialisation.
    Therefore ${(f_{t}, g_{t})}_t$ have a unique stationary distribution
    $(F_\infty, G_\infty)$.
\end{proof}

\subsection{Proof of \autoref{prop:convergence_approx}}\label{sec:proof_convergence_approx}

For presentation purpose, we first show that the ``slowed-down'' online Sinkhorn algorithm converges in the absence of noise. We then turn to prove \autoref{prop:convergence_approx}.

\subsubsection{Noise-free online Sinkhorn}

\begin{proposition}\label{prop:deterministic}
    We suppose that $\hat \alpha_t = \alpha$, $\hat \beta_t = \beta$ for all
    $t$. Then the updates \eqref{eq:updates} yields a (deterministic) sequence $(f_t, g_t)_t$ such
    that 
    \begin{equation}
        \Vert \hat f_t - f^\star \Vert_{\text{var}} 
        + \Vert \hat g_t - g^\star \Vert_{\text{var}} \to 0,\qquad
        \frac{1}{2} \dotp{\alpha}{f_t + \Ctrans{\hat g_t}{\alpha}} + \dotp{\beta}{\hat g_t + \Ctrans{f_t}{\beta}} 
         \to \Ww(\alpha, \beta).
    \end{equation}
\end{proposition}
Note that, as we perform \textit{simultaneous} updates, we only obtain
the convergence of $f_t \to f^\star + A$, and $g_t \to g^\star$, where $f^\star$
and $g^\star$ are solutions of \eqref{eq:wass} and $A$ is a constant depending
on initialization.

The \enquote{slowed-down} Sinkhorn iterations converge toward an optimal
potential couple, up to a constant factor: this stems from the fact that we
apply contractions in the space $(\Cc(\Xx), {\Vert\cdot\Vert}_{\var})$ with a
contraction factor that decreases sufficiently slowly.

\begin{proof}
    We write ${(f_t, g_t)}_t$ the sequence of iterates. Given a pair of optimal potentials 
    $(f^\star, g^\star)$, we write $u_t \triangleq f_t - f^\star$, $v_t \triangleq g_t - g^\star$,
    $u_t^T \triangleq \Ctrans{f_t}{\alpha} - g^\star$ and $v_t^T \triangleq \Ctrans{g_t}{\alpha} - f^\star$.
    For all $t > 0$, we observe that 
    \begin{align}
        \max u_{t+1} &= - \log \min \exp(-u_{t+1}) \\
        &= - \log \big( \min \big( (1 - \eta_t) \exp(-u_{t}) + \eta_t 
        \exp(-v_t^T) \big) \big)\\
        &\leq - \log \big( (1 - \eta_t) \min \exp(-u_{t}) + \eta_t 
        \min \exp(-v_t^T) \big)\\
        &\leq - (1 - \eta_t) \log \min \exp(-u_{t}) -  \eta_t \log \min
         \exp(-v_t^T) \\
         &= (1 - \eta_t) \max u_t  + \eta_t \max v_t^T,
    \end{align}
    where we have used the algorithm recursion on the second line, $\min f + g \geq \min f + \min g$
     on the third line and Jensen inequality on the fourth line. Similarly
    \begin{equation}
        \min u_{t+1} \geq (1 - \eta_t) \min u_t  + \eta_t \min v_t^T,
    \end{equation}
    and mirror inequalities hold for $v_t$. Summing the four inequalities, we obtain
    \begin{align}\label{eq:et}
        e_{t+1} &\triangleq \Vert u_{t+1} \Vert_{\var} + \Vert v_{t+1} \Vert_{\var} \\ 
        &= \max u_{t+1} - \min u_{t+1} + \max v_{t+1} - \min v_{t+1} \\
        &\leq
        (1 - \eta_t) ( \Vert u_t \Vert_{\var} + \Vert v_t \Vert_{\var})
        + \eta_t ( \Vert u_t^T \Vert_{\var} + \Vert v_t^T \Vert_{\var}), \\
        &\leq
        (1 - \eta_t) ( \Vert u_t \Vert_{\var} + \Vert v_t \Vert_{\var})
        + \eta_t \kappa ( \Vert u_t \Vert_{\var} + \Vert v_t \Vert_{\var}),
    \end{align}
    where we use the contractivity of the soft-$C$-transform, that guarantees that
    there exists $\kappa < 1$ such that $\Vert v_t^T\Vert_{\var} \leq \kappa \Vert
    v_t\Vert_{\var}$ and $\Vert u_t^T\Vert_{\var} \leq \kappa \Vert
    u_t\Vert_{\var}$ \citep{peyre2019computational}.

    Unrolling the recursion above, we obtain
    \begin{equation}
        \log e_t = \sum_{s=1}^t \log(1 - \eta_t (1 - \kappa)) + \log(e_0) \to - \infty,
    \end{equation}
    provided that $\sum \eta_t = \infty$. The proposition follows.
\end{proof}

\begin{proof}[Proof of \autoref{prop:convergence_approx}]
For discrete realizations $\hat \alpha$ and $\hat \beta$, we define the perturbation terms
\begin{equation}
    \epsilon_{\hat \beta}(\cdot) \triangleq
    f^\star - \Ctrans{g^\star}{\hat \beta} ,\qquad
    \iota_{\hat \alpha}(\cdot) \triangleq 
    g^\star - \Ctrans{f^\star}{\hat \alpha},
\end{equation}
so that the updates can be rewritten as
\begin{align}
    \exp(-f_{t+1} + f^\star) &= (1 - \eta_t)
    \exp(-f_{t} + f^\star)
    + \eta_t \exp(-\Ctrans{g_t}{\hat \beta_t} 
    +\Ctrans{g^\star}{\hat \beta_t} + \epsilon_{\hat \beta_t}) \\
    \exp(-g_{t+1} + g^\star) &= (1 - \eta_t)
    \exp(-g_{t} + g^\star)
    + \eta_t \exp(-\Ctrans{f_t}{\hat \alpha_t} 
    +\Ctrans{f^\star}{\hat \alpha_t} + \iota_{\hat \alpha_t}).
\end{align}
We denote $u_t \triangleq -f_{t} + f^\star$, $v_t \triangleq -g_{t} + g^\star$, $u_t^T \triangleq
\Ctrans{f_t}{\hat \beta_t} - \Ctrans{f^\star}{\hat \beta_t}$, $v_t^T \triangleq
\Ctrans{g_t}{\hat \beta_t} - \Ctrans{g^\star}{\hat \beta_t}$. Reusing the same
derivations as in the proof of \autoref{prop:deterministic}, we obtain
    \label{eq:pre_ineq_var}
    \begin{align}
    \Vert u_{t+1} \Vert_{\var} &\leq
    (1 - \eta_t) \Vert u_t \Vert_{\var}
    \\
    &\phantom{=}+ \eta_t \log \big( \max_{x, y \in \Xx}
    \exp(\epsilon_{\hat \beta_t}(x) 
    - \epsilon_{\hat \beta_t}(y)) \exp(v_t^T(x) - v_t^T(y)) \big) \\ 
    &\leq
    (1 - \eta_t) \Vert u_t \Vert_{\var}
    + \eta_t \Vert v_t^T \Vert_{\var}
    + \eta_t \Vert \epsilon_{\hat \beta_t} \Vert_{\var},
\end{align}
where we have used $\max_x f(x) g(x) \leq \max_x f(x) \max_x f(x)$ on the second line. Therefore,
using the contractivity of the soft $C$-transform,
\begin{equation}
    \label{eq:ineq_var}
    e_{t+1} \leq 
    (1 - \tilde \eta_t) e_t
    +  \frac{\tilde \eta_t}{1 -\kappa}
    ({\Vert \epsilon_{\hat \beta_t} \Vert}_{\var} + 
    {\Vert \iota_{\hat \alpha_t} \Vert}_{\var}),
\end{equation}
where we set $e_t \triangleq \Vert u_t \Vert_{\var} + \Vert v_t \Vert_{\var}$,
$\tilde \eta_t = \eta_t (1-\kappa)$ and $\kappa$ is set to be the biggest
contraction factor over all empirical realizations $\hat \alpha_t$, $\hat
\beta_t$ of the distributions $\alpha$ and $\beta$. It is upper bounded by $1 -
e^{- L\text{diam}(\Xx)}$, thanks to \autoref{ass:lip} and
\autoref{lemma:contractivity}.

The realizations $\hat \beta_t$ and $\hat \alpha_t$
are sampled according to the same distribution $\hat \alpha$ and $\hat \beta$. We
define the sequence $r_t$ to be the running average of the variational norm of the
(functional) error term:
\begin{equation}
    r_{t+1} \triangleq (1 - \tilde \eta_t) r_t + \frac{\tilde \eta_t}{1 - \kappa}
    ({\Vert \epsilon_{\hat \beta_t} \Vert}_{\var} + 
    {\Vert \iota_{\hat \alpha_t} \Vert}_{\var}).
\end{equation}
We thus have, for all $t > 0$, $e_t \leq r_t$. Using \autoref{lemma:running}, the sequence $(r_t)_t$
converges towards the scalar expected value
\begin{equation}\label{eq:r_infty}
    r_\infty \triangleq \frac{1}{1 - \kappa} \EE_{\hat \alpha, \hat \beta}[{\Vert \epsilon_{\hat \beta} \Vert}_{\var}
    + {\Vert \iota_{\hat \alpha} \Vert}_{\var}] > 0.
\end{equation}
We now relate $r_\infty$ to the number of samples $n$ using
a uniform law of large number result on parametric functions.
We write $\hat \beta = \hat \beta_n$ to make explicit the dependency of the
quantities on the batch size $n$.

Using \autoref{lemma:lln}, we bound the quantity
\begin{align}
    E_n &\triangleq \EE_{\hat \beta_n} {\Vert \epsilon_{\hat \beta_n} \Vert}_{\var}
     = \EE_{\hat \beta_n} {\Vert \exp(-\Ctrans{g^\star_0}{\beta})
    - \exp(-\Ctrans{g^\star_0}{\hat \beta_n}) \Vert}_{\infty} \\
    &= \EE_{Y_1, \dots Y_n \sim \beta} 
    \sup_{x \in \Xx}
     \Big| \frac{1}{n} \sum_{i=1}^n \exp(g^\star(Y_i)) - C(x, Y_i)) \\
      &\phantom{=}\qquad\qquad\qquad\quad- \EE_{Y \sim \beta}[\exp(g^\star_0(Y)) - C(x, Y)] \Big| \\
      &= \EE \sup_{x \in \Xx} | \frac{1}{n} \sum_{i=1}^n \phi_i(x) - \phi(x) | ,
\end{align}
where we have defined $\phi_i: x \to \exp(g^\star(Y_i) - C(x, Y_i))$ and set
$\phi$ to be the expected value of each $\phi_i$. The compactness of $\Xx$
ensures that the functions  are square integrable and uniformly bounded.
\autoref{lemma:lln} ensures that there exists $S(g^\star)$ such that
\begin{equation}
    E_n \leq \frac{S(g^\star)}{\sqrt{n}}.
\end{equation}
We now bound $\EE_{\hat \beta_n} {|| \epsilon_{\hat \beta_n} ||}_\var$ using the
 quantity $E_n$. First, we observe that $\Vert_{\var} = g^\star_{\min} < g^\star
 < 0$, and there exists $C_{\max} > 0$ such that $0 \leq C(x, y) \leq C_{\max}$
 for all $x, y \in \Xx$, thanks to the \autoref{ass:lip}.
 \begin{align}
    \delta \triangleq \exp(-\Vert g^\star \Vert_{\var}
     - C_{\max}) &\leq \exp(-T_\beta(g^\star)) \leq 1 \\
    \exp(-\Vert g^\star \Vert_{\var}
     - C_{\max}) &\leq \exp(-T_{\hat \beta_n} 
    (g^\star)) \leq 1,
\end{align}
where we have used $g^\star = \Vert g^\star \Vert_{\var}$.
For all $x \in \Xx$, 
\begin{equation}\label{eq:expression}
    | \epsilon_{\hat \beta_n} | = 
    | \log \frac{\exp(-T_{\hat \beta_n}(g^\star))}{\exp(-T_\beta(g^\star))} | =
    \Big| \log\big(1 + 
    \frac{\exp(-\Ctrans{g^\star}{\hat \beta_n})
    -\exp(-\Ctrans{g^\star}{\beta})
    }
    {\exp(-\Ctrans{g^\star}{\beta})}\big) \Big|.
\end{equation}
We first obtain an upper-bound independent of $n$ with the first equality in
 \eqref{eq:expression}:
\begin{equation}\label{eq:const_ineq}
    {|| \varepsilon_{\hat \beta_n} ||}_{\var} \leq {|| \varepsilon_{\hat \beta_n} ||}_{\infty} \leq 
    {\Vert g^\star \Vert}_{\var} + C_{\max}.
\end{equation}
 We now use the second expression in \eqref{eq:expression}: for $n$ large enough, $E_n < \delta$
\begin{equation}\label{eq:log_ineq}
    {|| \varepsilon_{\hat \beta_n} ||}_{\var} \leq \max ( \log(1 + \frac{E_n}{\delta}),
     - \log(1 - \frac{E_n}{\delta}) ) = 
    - \log(1 - \tilde E_n),
\end{equation}
where we have set $\tilde E_n \triangleq \frac{E_n}{\delta}$. On the event
$\Omega_n = \{\tilde E_n \leq \frac{1}{2}\}$, a simple calculation gives $- \log(1
- \tilde E_n) \leq (2 \log 2) \tilde E_n \leq 2 \tilde E_n$. Thanks to Markov inequality,
$\PP[\tilde E_n > \frac{1}{2}] \leq 2 \EE[\tilde E_n]$. We then split the
expectation over the event $\Omega_n$, and use inequalities \eqref{eq:log_ineq}
and \eqref{eq:const_ineq} on each conditional expectation:
\begin{align}\label{eq:vdv}
    \EE {|| \varepsilon_{\hat \beta_n} ||}_{\var}  &= \PP\left[\tilde E_n \leq \frac{1}{2}\right] 
    \EE \left[ {|| \varepsilon_{\hat \beta_n} ||}_{\var}
    \Big| \tilde E_n \leq \frac{1}{2}\right] 
    \\
    &\phantom{=}
    + \PP\left[\tilde E_n > \frac{1}{2}\right]      \EE \left[ {|| \varepsilon_{\hat \beta_n} ||}_{\var}
    \Big| \tilde E_n > \frac{1}{2}\right] \\
    &\leq \frac{2 \phi({\Vert g^\star \Vert}_{\var} + C_{\max}) 
    S(g^\star)}{\sqrt{n}}\\
    &\leq \frac{4 \exp({\Vert g^\star \Vert}_{\var} + C_{\max}) 
    S(g^\star)}{\sqrt{n}} \triangleq \frac{A(g^\star)}{\sqrt{n}}
\end{align}
The constants $S$ depends on the complexity of estimating
the functional $x \to \int_y \exp(g^\star(y) - C(x,y)) \d \beta(y)$ with samples from $\beta$.
A parallel
result holds for $\EE_{\hat \alpha_n} {\Vert \iota_{\hat \alpha_n}
\Vert}_{\var}$. Therefore, there exists $A(f^\star), A(g^\star) > 0$ such that $r_\infty \leq
\frac{A(f^\star) + A(g^\star)}{\sqrt{n}}$. As for all $t >0$, $e_t \leq r_t \to_{t \to \infty}
r_\infty$, the proposition follows,  writing $A = A(f^\star) + A(g^\star)$.

The constant $A$ is larger than $\exp(C_{\max})$ when $C_{\max} \to
\infty$; Hence it behaves at least like $\exp(\frac{1}{\varepsilon})$ when~$\varepsilon
\to 0$.

Note that we have used twice a corollary of the law of large numbers: once when
averaging over $t$ with $t \to \infty$ (Eq. \eqref{eq:r_infty}), and once when
averaging over $n$ with $n$ finite (Eq. \eqref{eq:vdv}).
\end{proof}

\subsection{Proof of \autoref{prop:convergence_true}}\label{sec:proof_prop_convergence}

In the proof of \autoref{prop:convergence_approx} and in particular Eq. \eqref{eq:ineq_var}, the term that prevents the convergence of $e_t$ is
\begin{equation}
    \eta_t ({\Vert \epsilon_{\hat \beta_t} \Vert}_{\var} + 
    {\Vert \iota_{\hat \alpha_t} \Vert}_{\var}),
\end{equation}
which is not summable in general. We can control this term by increasing the
 size of $\hat \alpha_t$ and $\hat \beta_t$ with time, at a sufficient rate: this is what \autoref{ass:double_weights} ensures.

\begin{proof}
From Eq. \eqref{eq:ineq_var}, for all $t > 0$, we have
\begin{equation}\label{eq:Et_rec}
    0 \leq e_{t+1} \leq 
    (1 - \tilde \eta_t) e_t
    + \eta_t
    ({\Vert \epsilon_{\hat \beta_t} \Vert}_{\var} + 
    {\Vert \iota_{\hat \alpha_t} \Vert}_{\var}).
\end{equation}
Taking the expectation and using the uniform law of 
large number \eqref{eq:vdv},
\begin{align}\label{eq:recursion}
    \EE e_{t+1} &\leq (1 - (1 - \kappa) \eta_t) \EE e_t + 
    \eta_t \frac{A}{\sqrt{n(t)}} \\
    &= (1 - (1 - \kappa) \eta_t) \EE e_t + 
    A \eta_t w_t,
\end{align}
where we have used the definition of $n(t)$ from \autoref{ass:double_weights} in the last line.

The proof follows from a simple asymptotic analysis of the sequence ${(\EE e_t)}_t$, following recursion \eqref{eq:recursion}.
For all $t > 0$, 
\begin{equation}\label{eq:Et_rec2}
    \EE e_{t+1} - \EE e_t = - (1 - \kappa) \eta_t \EE e_t + A \eta_t w_t
     \leq A \eta_t w_t
\end{equation}
Therefore, from \autoref{ass:double_weights}, ${(\EE e_{t+1} - \EE e_t)}_t$ is
summable and $\EE e_t \to_{t \to \infty} \ell \geq 0$. Let's assume $\ell > 0$.
Summing \eqref{eq:Et_rec2} over $t$, we obtain
\begin{equation}
    \EE e_t \leq \EE e_1 - (1 - \kappa) \sum_{s=1}^{t-1} \eta_s \EE_s 
    + A \sum_{s=1}^{t-1} \eta_s w_s \to_{t \to \infty} - \infty,
\end{equation}
which leads to a contradiction. Therefore $\EE e_t \to_{t \to \infty} 0$. As
$e_t \geq 0$ for all $t > 0$, this implies that $e_t \to_{t \to \infty} 0$
almost surely.
\end{proof}

\subsection{Proof of \autoref{prop:rate}}\label{app:proof_rate}

\begin{proof}
The proof of \autoref{prop:convergence_true} allows us to derive non-asymptotic rates for potential estimations using the online Sinkhorn algorithm. Let us set $\eta_t = \frac{\lambda}{t^a}$, $n(t) = \lceil B t^{2b} \rceil$ in \eqref{eq:Et_rec}, so that \autoref{ass:double_weights} is met.
$\lceil \cdot \rceil$ denotes the ceiling function.
We are left to study the recursion \eqref{eq:recursion}:
\begin{equation}
    \delta_{t+1} \triangleq \EE e_{t+1} \leq (1 - \frac{\lambda(1 - \kappa)}{t^a}) \delta_t + 
    \frac{A \lambda}{\sqrt{B}{t^{a + b}}}
\end{equation}

Following the derivations of \citet[Theorem 2]{moulines_non-asymptotic_2011}, we have the following
bias-variance decomposed upper-bound,
provided that $0 \leq a < 1$ and $a+ b > 1$. For all $t > 0$,
\begin{equation}\label{eq:rates}
    \delta_t \leq (\delta_0 + \frac{A S}{(a + b - 1)\sqrt{B}})
    \exp(- \frac{S(1 - \kappa)}{2} t^{1 - a})
    + \frac{2 A S}{\sqrt{B}(1 - \kappa) t^a}.
\end{equation}
Let us now relate the iteration number $t$ to the number of seen sample $N$. By definition
\begin{equation}
    n_t = \sum_{s=1}^t n(s) \leq B \sum_{s=1}^t s^{2b} + t \leq
     t + \frac{(t+1)^{2b + 1} - 1}{2b + 1}
     \leq (2t)^{2b+1}.
\end{equation}
Therefore, when we have seen $N$ samples, the iteration number is superior to $t(N)$, and the expected error $\delta_N$ is of the order of $\delta_{t(N)}$, with
\begin{equation}\label{eq:tn}
    t(N) =  {(N/2)}^{\frac{1}{2b + 1}}.
\end{equation}
We write $\delta_N = \delta_{t(N)}$. Replacing \eqref{eq:tn} in \eqref{eq:rates} yields
\begin{equation}\label{eq:non-asymptotic}
    \delta_n \leq 
    (\delta_0 + \frac{A \lambda}{(a + b - 1)\sqrt{B}})
    \exp\left(- \frac{\lambda(1 - \kappa)}{2} {(n /2)}^{\frac{1 - a}{2b+1}}\right)
    + \frac{2 A \lambda}{\sqrt{B}(1 - \kappa) {(n/2)}^{\frac{a}{2b+1}}}.
\end{equation}
We note that $b$ and $a$ should be as close to $0$ as possible to reduce the bias term, while $a$
should be as close to $1$ and $b$ as close to $0$ as possible to reduce the variance
term. Of course, $b$ should remain larger than $1 - a$ to ensure convergence.

To obtain the best asymptotical rates (the error is always dominated by the variance term), we set $a = 1 - \iota$, $ b = 2 \iota$, with $\iota \succsim 0$. This yields
\begin{align}\label{eq:D_constant}
    \delta_n &\leq 
    (\delta_0 + \frac{A \lambda}{\iota \sqrt{B}})
    \exp\left(- \frac{\lambda(1 - \kappa)}{2} {(n/2)}^{\frac{\iota}{1 + 4\iota}}\right)
    + \frac{2 A \lambda}{\sqrt{B}(1 - \kappa) {(n/2)}^{\frac{1 - \iota}{1 + 4\iota}}} \\
    &= \Oo(n^{-\frac{1 - \iota}{1 + 4 \iota}}). 
\end{align}

This rate is as close to the rate $\Oo(\frac{1}{n})$ as desired. We may then
perform a last soft $C$-transform (using the $n_t$ seen samples) over the
estimated $f_{t(n)}, g_{t(n)}$ to obtain a estimated solution of the dual
optimisation problem \eqref{eq:dual}. The Sinkhorn potentials can therefore be
estimated with \textit{fast rates}. Note that the upper bound explodes when $\varepsilon
\to 0$, as $C_{\max} \to \infty$, hence $A \to \infty$, and $(1 - \kappa) \to 0$.
\end{proof}

\paragraph{Estimating the Sinkhorn distance.} The Sinkhorn distance requires to estimate the integral
\begin{equation}
    \Ww(\alpha, \beta) = \int_x f^\star(x) \d \alpha(x) + \int_y g^\star(y) \d \beta(y).
\end{equation}
At iteration $t(n)$, with empirical realization $\bar \alpha_t$ and $\bar
\beta_t$, containing $n$ samples, we use the estimator
\begin{equation}
    \hat \Ww(\alpha, \beta) = \frac{1}{n} \sum_{i=1}^n f_{t(n)}(x_i) + \frac{1}{n} \sum_{i=1}^n g_{t(n)}(y_i),
\end{equation}
We can bound the estimation error $|\hat \Ww(\alpha, \beta) - \Ww(\alpha, \beta)
| = \Oo(\frac{1}{\sqrt{n}})$, dominated by the integral evaluation noise. We
thus recover a new estimator of the Sinkhorn distance with the same sample
complexity as the batch Sinkhorn estimator \citep{2019-Genevay-aistats}. Our
estimator enjoys an original rate for estimating the potentials in $\Vert \cdot
\Vert_{\var}$.

\pagebreak
\section{Online Sinkhorn variants}

\subsection{Fully-corrective scheme}\label{app:fully_corrective}

We report the fully-corrective online Sinkhorn algorithm in \autoref{alg:full_corrective}. This algorithm also enjoys almost sure convergence, provided that the following assumption is met.

\begin{algorithm}[t]
    \begin{algorithmic}
    \State \textbf{Input:} Distribution $\alpha$ and $\beta$, learning weights ${(\eta_t)}_t$ and batch-sizes ${(n(t))}_t$. 
    \textbf{Set} $p_{i,1} = q_{i,1} = 0$ for $i \in (0, n_1]$
    \For{$t= 0, \dots, {T-1}$}
        \State Sample $(x_i)_{(n_t, n_{t+1}]} \sim \alpha$, $(y_j)_{(n_t, n_{t+1}]} \sim \beta$.
            \State Evaluate $(\hat f_t(x_i))_{i=(0, n_{t+1}]}$,
             $(\hat g_t(y_i))_{i=(0, n_{t+1}]}$ using $(q_{i,t}, p_{i,t}, x_i, y_i)_{i=(0,n_{t}]}$ in \eqref{eq:param}.
             \State $q_{(0, n_{t+1}],t+1} {\gets} \log \frac{1}{n}
             + (\hat g_t(y_i))_{(0, n_{t+1}]}$,
             \qquad $p_{(n_t, n_{t+1}],t+1} {\gets} \log \frac{1}{n} 
             + (\hat f_t(x_i))_{(n_t, n_{t+1}]}$.
    \EndFor
    \State \textbf{Returns:} $\hat f_T : (q_{i,T}, y_i)_{(0, n_T]}$ and
    $\hat g_T : (p_{i,T}, x_i)_{(0, n_T]}$
    \end{algorithmic}
    \vspace{-.4em}
    \caption{Fully-corrective online Sinkhorn}\label{alg:full_corrective}
\end{algorithm}

\begin{algorithm}[t]
    \begin{algorithmic}
    \Input Distribution $\alpha \in \triangle^N$ and 
    $\beta \in \triangle^N$, $x \in \RR^{n \times d}$, 
    $y \in \RR^{n \times d}$, learning weights ${(\eta_t)}_t$
    \State \textbf{Set} $p = q = - \infty \in \RR^n$.
    \For{$t= 1, \dots, {T}$}
        \State $q \gets q + \log(1 - \eta_t)$, $p \gets p + \log(1 - \eta_t)$.
        \State Sample $J_t \subset [1, N]$, $I_t \subset [1, N]$ of size $n(t)$.
        \For{$i \in J_t$}
            \State $q_i \gets \log \Big( \exp(q_i)
            + \exp\big(\log(\eta_t) - \log \frac{1}{N} 
            \sum_{j=1}^{N} \exp(p_j - C(x_j, y_i)\big) \Big) $.
        \EndFor
        \For{$i \in I_t$}
        \State $p_i \gets \log \Big( \exp(q_i)
        + \exp \big( \log(\eta_t) - \log \frac{1}{M} 
        \sum_{j=1}^{M} \exp(q_j - C(x_i, y_j)\big) \Big)$.
        \EndFor
    \EndFor
    \State Returns $f_T : (q, y)$ and
    $g_T : (p, x)$
    \end{algorithmic}
    \caption{Online Sinkhorn potentials in the discrete setting}\label{alg:discrete_online}
\end{algorithm}

\begin{assumption}\label{ass:total_growing}
    For all $t > 0$, the \textit{total} batch-size $n_t = \frac{B}{w_t^2}$ is an
        integer. The step-size $\eta_t$ and the batch-size $n_t$ grows so that
        $\sum w_t \eta_t < \infty$ and $\sum \eta_t = \infty$.
\end{assumption}

With full correction, the total number of observed samples $n_t$ needs to grow
at the same rate as the single-iteration batch-size $n(t)$ in
\autoref{ass:double_weights}. For $\eta_t = \frac{1}{t^a}$, $a \in (1/2, 1]$, it
is sufficient to use a constant batch-size $n(t) = B$ to meet \autoref{ass:total_growing}. We then have the following property

\begin{proposition}
    Under \autoref{ass:lip} and
    \ref{ass:total_growing}, the fully-corrective online Sinkhorn algorithm converges almost surely:
    \begin{equation}
        \Vert \hat f_t - f^\star \Vert_{\var} + \Vert \hat g_t - g^\star \Vert_{\var} \to 0.
    \end{equation}
\end{proposition}

\begin{proof}
    Using the fully-corrective scheme allows to replace $n(t)$ by $n_t =
    \sum_{s=0}^t n(s)$ in \eqref{eq:recursion}. The proposition is then obtained
    in the same way as \autoref{prop:rate}.
\end{proof}

\subsection{Online Sinkhorn for discrete distributions}\label{sec:sinkhorn_discrete}

The online Sinkhorn algorithm takes a simpler form with discrete
distributions. We derive it in \autoref{alg:discrete_online}. We set $\alpha$
and $\beta$ to have size $N$ and $M$, respectively. We evaluate
the potentials as
\begin{align}
    g_t(y) &= - \log \sum_{j=1}^N \exp(p_j - C(x_j, y)) \\
    f_t(x) &= - \log \sum_{j=1}^M \exp(q_j - C(x, y_j)), \\
\end{align}

\vspace{-2em}
where $(p_j)_{J \in [1, N]}$ and $(q_j)_{J \in [1, M]}$ are fixed-size vectors.
Note that the computations written in \autoref{alg:discrete_online} are
in log-space,as they should be implemented to prevent numerical overflows. The sets $|I|$ and $|J|$ can
have varying sizes along the algorithm, which allows for example to speed-up the
initial Sinkhorn iteration (\autoref{sec:accelerating}). In this case, the
cost matrix $\hat C = C(x_i,y_j))_{i,j}$ should be progressively recorded along the algorithm iterations.

\if\icml1
\subsection{Recapitulation on batch-sizes and learning rates}

To provide practical guidance on choosing rates in batch-sizes $n(t)$ and
step-sizes $\eta_t$, we can parametrize $\eta_t = \frac{1}{t^a}$ and $n(t) = B
t^b$ and study what is implied by \autoref{ass:double_weights} and
\autoref{ass:total_growing}. We summarize the schedules for which convergence is
guarantees in \autoref{table:growing}. Note that in practice, it is useful to
replace $t$ by $(1 + r\, t)$ in these schedules. We set $r=0.1$ in all
experiments.
\setcounter{table}{1}
\begin{table}[t]
    \centering
    \caption{Schedules of batch-sizes and learning rates that ensures online Sinkhorn convergence.}

    \begin{tabular}{ccc}
        \toprule
        Param. schedule &Online Sinkhorn & Fully-corrective online Sinkhorn \\
        \midrule
        Batch size $\displaystyle n(t) = B t^b$ & $0 < b $ & $0 \leq b$ \\
        Step size $\displaystyle \eta_t = \frac{1}{t^a}$ &$\displaystyle a \geq 1 - \frac{b}{2}$ & 
        \parbox{5cm}{
        \begin{equation}
            \left\{
                \begin{aligned}
            a &> \frac{1}{2} - \frac{b}{2} \qandq b <1 \\
            a &\geq 0 \qandq b \geq 1
                \end{aligned}
                \right.
        \end{equation}}\\
        \bottomrule
    \end{tabular}
    \label{table:growing}
\end{table}
\fi
\vfill
\pagebreak

\section{Extra numerical experiments}

We display and describe the supplementary figures mentionned in the main text, as well as experimental details useful for reproduction.

\subsection{Online Sinkhorn and variants}\label{app:online_exp}

\begin{figure}[t]
    \begin{widepage}
        
    \centering
    $\varepsilon = 0.1$

    \includegraphics[width=\linewidth]{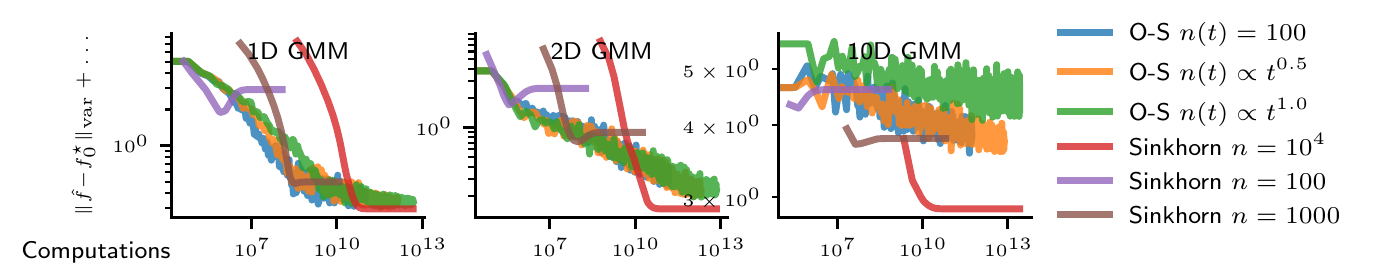}
    
    $\varepsilon = 0.01$

    \includegraphics[width=\linewidth]{online_0_01_False_test.pdf}

    $\varepsilon = 0.001$

    \includegraphics[width=\linewidth]{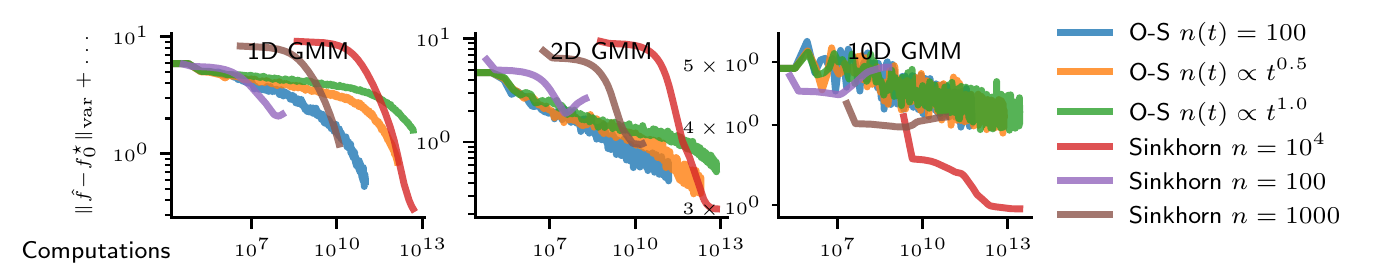}

    $\varepsilon = 0.0001$

    \includegraphics[width=\linewidth]{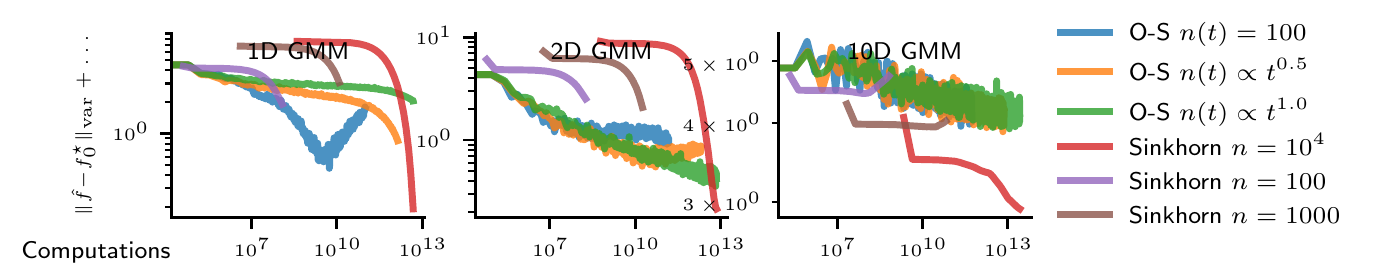}
    \end{widepage}
    \caption{Performance of online Sinkhorn for various $\varepsilon$.}
    \label{fig:convergence_all}
\end{figure}

\begin{figure}[t]
    \begin{widepage}
    \centering
    $\varepsilon = 0.1$

    \includegraphics[width=\linewidth]{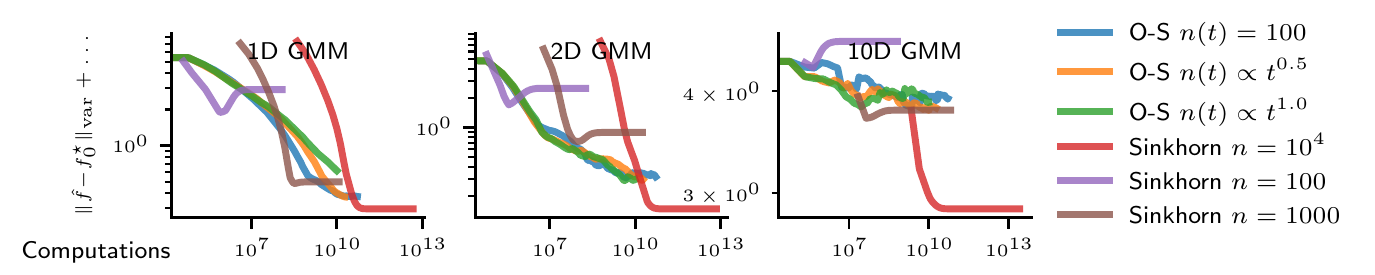}
    
    $\varepsilon = 0.01$

    \includegraphics[width=\linewidth]{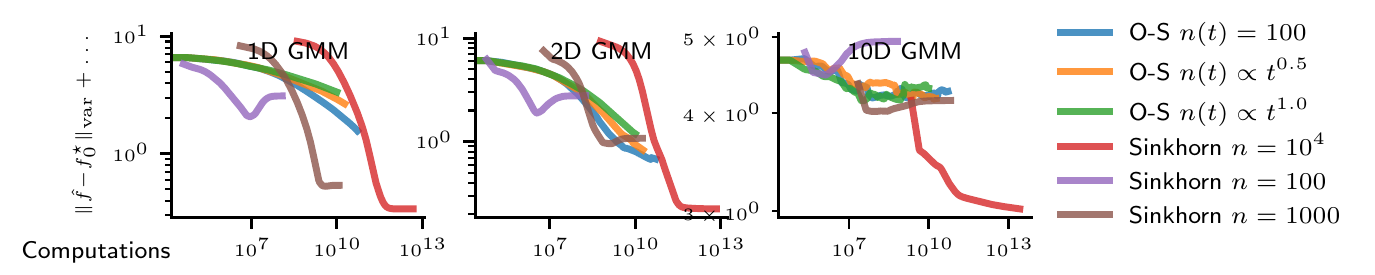}

    $\varepsilon = 0.001$

    \includegraphics[width=\linewidth]{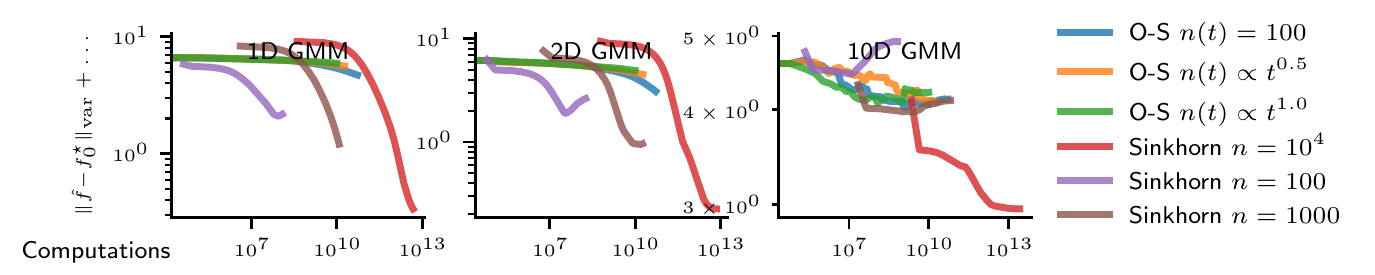}

    $\varepsilon = 0.0001$

    \includegraphics[width=\linewidth]{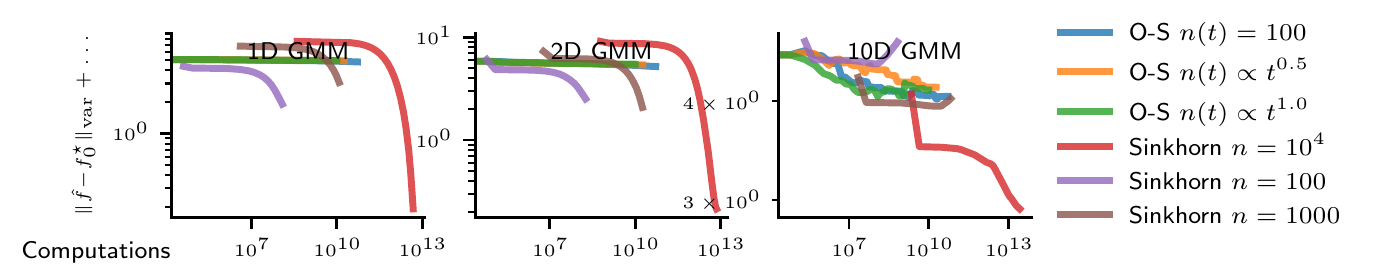}
    \end{widepage}

    \caption{Performance of fully-corrective online Sinkhorn (O-S) for various $\varepsilon$.}
    \label{fig:convergence_refit}
\end{figure}

\begin{figure}[t]
    \begin{widepage}
    \centering
    \includegraphics[width=\linewidth]{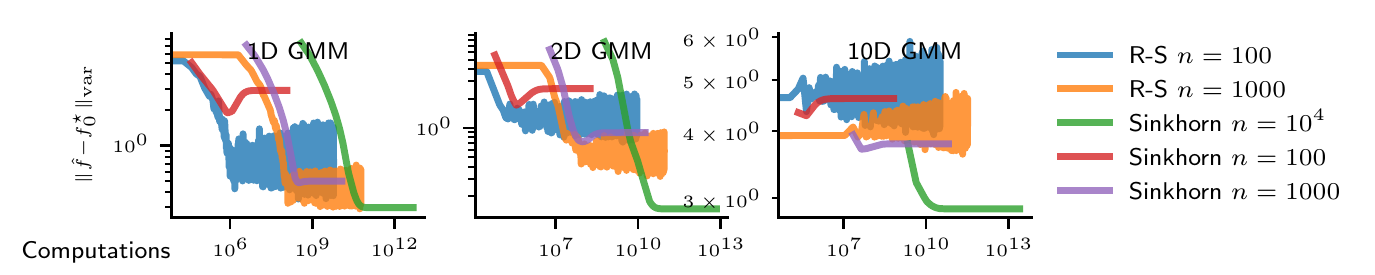}
    
    $\varepsilon = 0.01$

    \includegraphics[width=\linewidth]{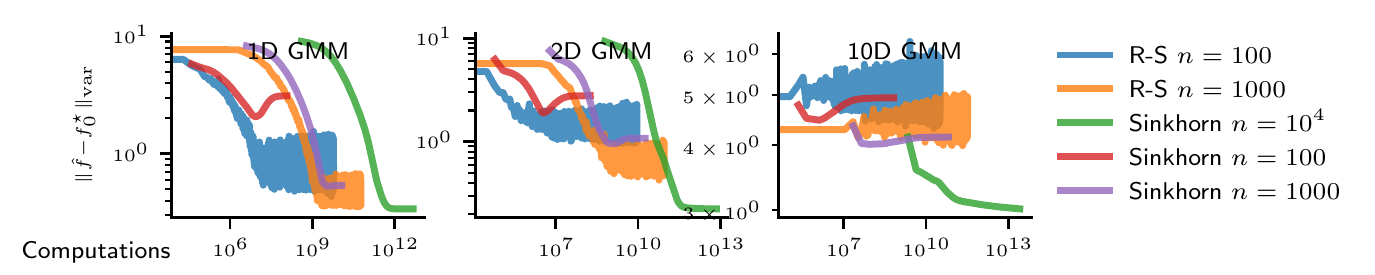}

    $\varepsilon = 0.001$

    \includegraphics[width=\linewidth]{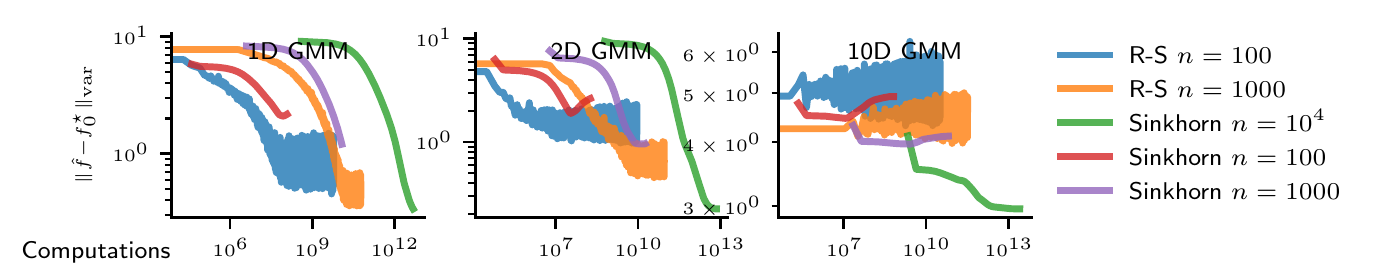}

    $\varepsilon = 0.0001$

    \includegraphics[width=\linewidth]{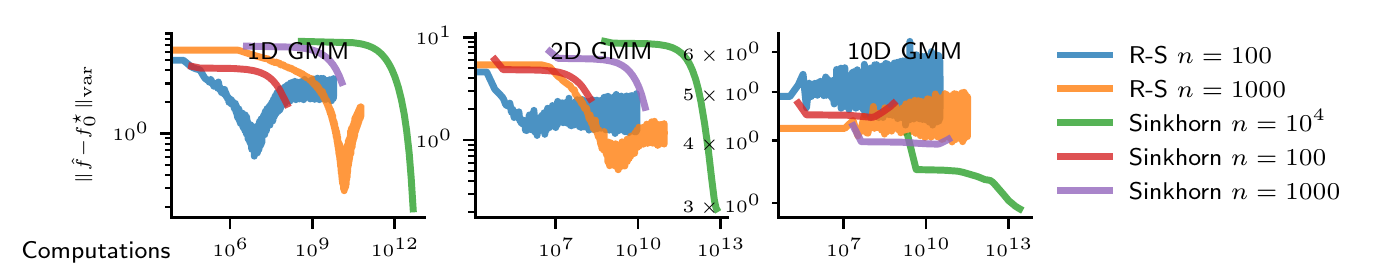}
    \end{widepage}

    \caption{Performance of randomized Sinkhorn (R-S) for various $\varepsilon$.}
    \label{fig:convergence_randomized}
\end{figure}

\paragraph{Grids and details for \autoref{sec:online_exp}.} We set $(\eta_t,
n(t)) = \big(\frac{1}{(1 + 0.1t)^a}, 100 (1 + 0.1t)^{b}\big)$, with $(a, b) =
(0, 2)$, $(a, b) = (\frac{1}{2}, 1)$ and $(a,b =1, 0)$ (constant batch-sizes).
Batch Sinkhorn algorithms uses $N=~100,1000,10000$. We train Sinkhorn on $t =
5000$ iterations, and train online Sinkhorn long enough to match the number of
computations of the large Sinkhorn reference. 

\paragraph{All OS convergence curves.} To complete \autoref{fig:convergence},
\autoref{fig:convergence_all} report the performance of online Sinkhorn for
$\varepsilon \in \{10^{-4}, 10^{-3}, 10^{-2}, 10^{-1}]\}$. The comparison of
performance remains similar to the one produced in the main text.

\paragraph{Fully-corrective online Sinkhorn.}\autoref{fig:convergence_refit} reports the performance of
fully-corrected online Sinkhorn (FCOS). We observe that the fully-corrective scheme is less noisy than the
non-corrected one. It is less efficient than OS on low-dimensional problems,
but faster on the 10 dimensional problem. For GMM-10D, it outperforms the batch Sinkhorn algorithm
with $N=100, 1000$. Note that we interrupt FCOS for $n_t > 20,000$, as our implementation of the $C$-transform has a quadratic memory cost in $n_t$---this cost can be reduced to a linear cost with more careful implementation \footnote{Using e.g. \url{https://www.kernel-operations.io/keops/index.html}}.

\paragraph{Randomized Sinkhorn.}\autoref{fig:convergence_randomized} reports the
performance of randomized Sinkhorn. In low dimension, randomized Sinkhorn is a
reasonable alternative to batch Sinkhorn, as it often outperforms it on average,
for the same memory complexity (compare purple to orange curve for instance). In
high dimension, batch Sinkhorn tend to perform slightly better.

\subsection{OT between Gaussians}

We measure the performance of online Sinkhorn to transport one Gaussian
distribution $\alpha$ to another $\beta$. The potentials $f^\star, g^\star$ are
known exactly for this problem, which allows to have a strong golden standard.
More precisely, adapting the formulae from \cite{janati_entropic_2020}, assuming
$\alpha \sim \Nn(\mu, A)$ and $\beta \sim \Nn(\nu, \beta)$ and writing $I$ the
identity matrix in $\RR^d$, we have
\begin{gather}
    C \triangleq (A B + \frac{\epsilon^2}{4} I)^{1/2}, \quad
    U \triangleq B (C + \frac{\varepsilon}{2} I)^{-1} - I,\quad
    V \triangleq A (C + \frac{\varepsilon}{2} I)^{-1} - I \\
    f^\star: x \to - \frac{1}{2} (x - \mu)^\top U (x - \mu) + x^\top (\mu - \nu) \\
    g^\star: y \to - \frac{1}{2} (y - \nu)^\top V (y - \nu) + y^\top (\nu - \mu) \\
\end{gather}

We compare batch Sinkhorn $(N=100, 1000, 10000)$ to (non fully-corrected) online
Sinkhorn, with $n(t) = B$, and $n(t) = B (1 + 0.1 t)^{1/2}$, $B=100$, and
$\varepsilon \in \{ 10^{-4}, 10^{-3}, 10^{-2}, 10^{-1}\}$.

\begin{figure}[t]
    \begin{widepage}
        \centering
        $\varepsilon = 0.1$
    
        \includegraphics[width=\linewidth]{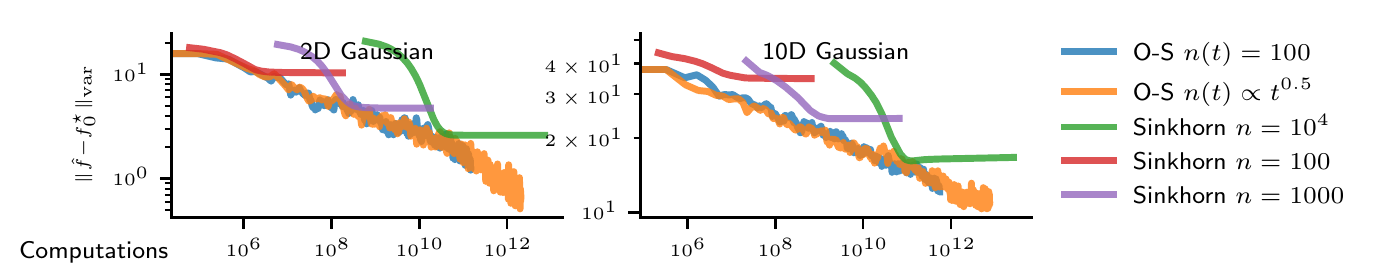}
        
        $\varepsilon = 0.01$

        \includegraphics[width=\linewidth]{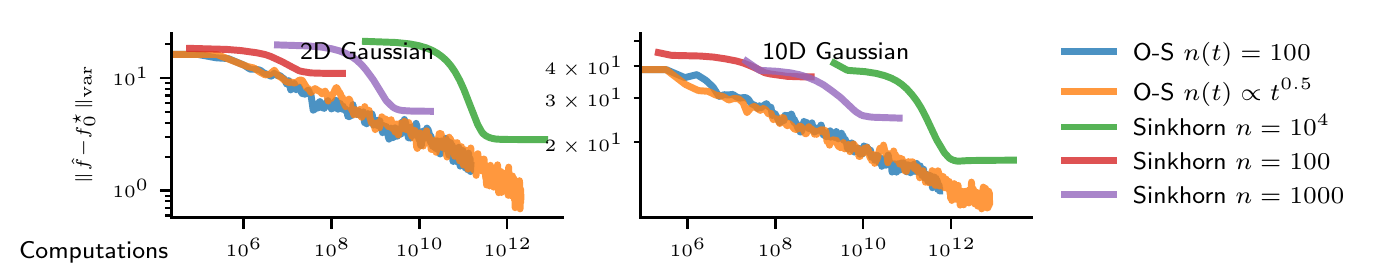}
    
        $\varepsilon = 0.001$
    
        \includegraphics[width=\linewidth]{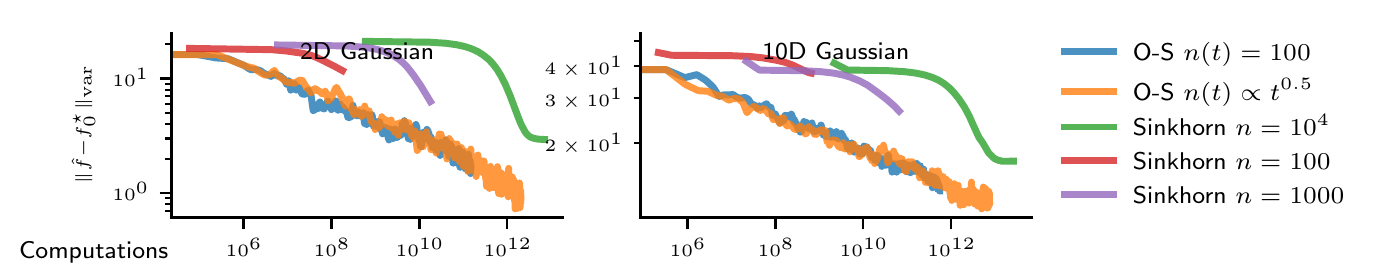}
    
        $\varepsilon = 0.0001$
    
        \includegraphics[width=\linewidth]{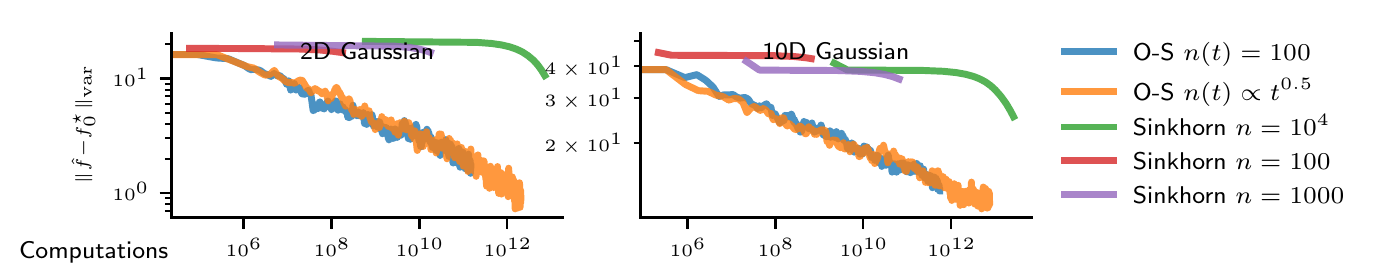}
    \end{widepage}
    \caption{Performance of online-Sinkhorn to estimate OT between two Gaussians. Online Sinkhorn systematically outperforms batch Sinkhorn, but in term of speed and correction.}
    \label{fig:gaussian}
\end{figure}

\begin{figure}[htbp]
    \centering
    
    \includegraphics[width=\linewidth]{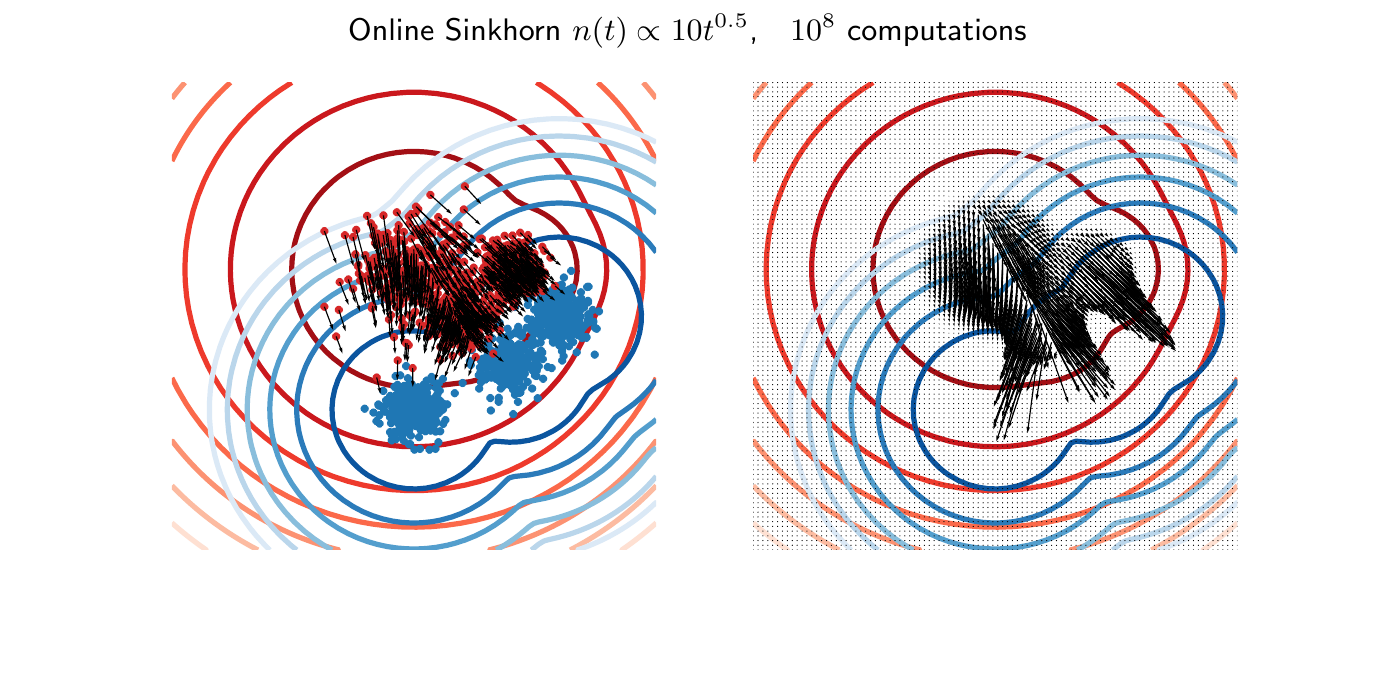}

    \includegraphics[width=\linewidth]{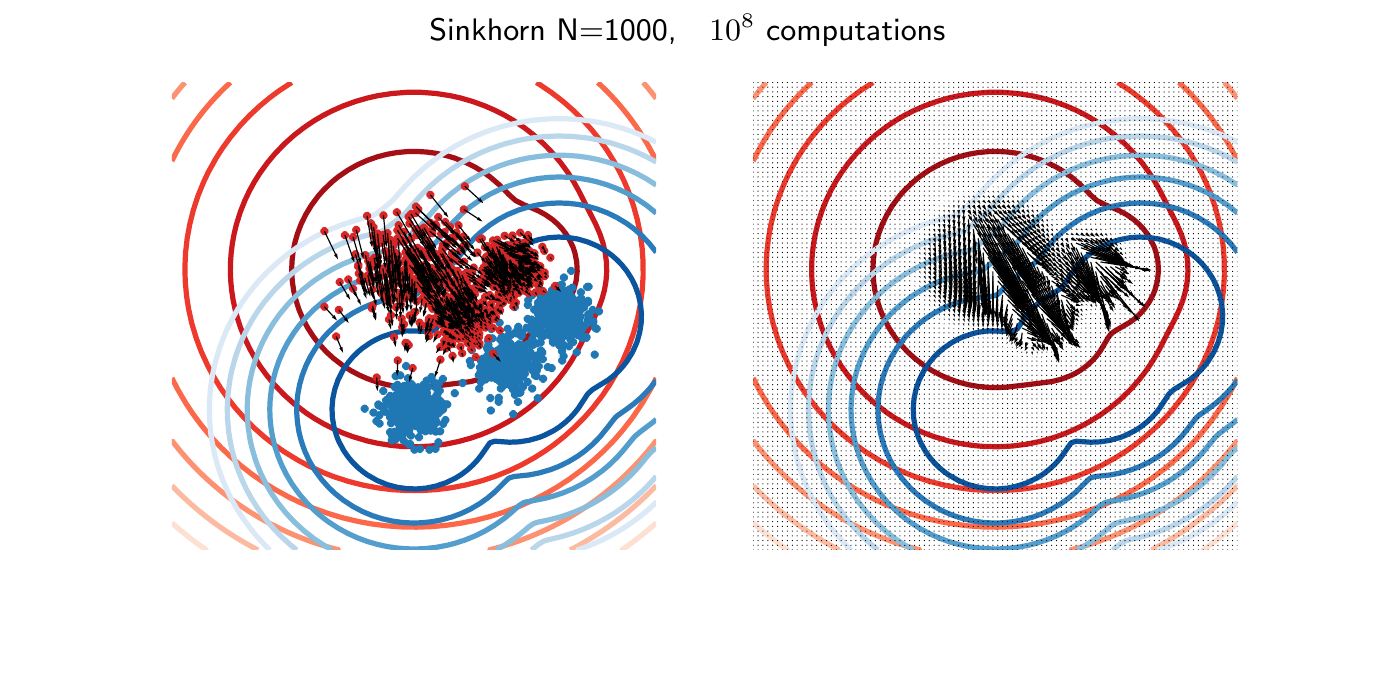}

    \includegraphics[width=\linewidth]{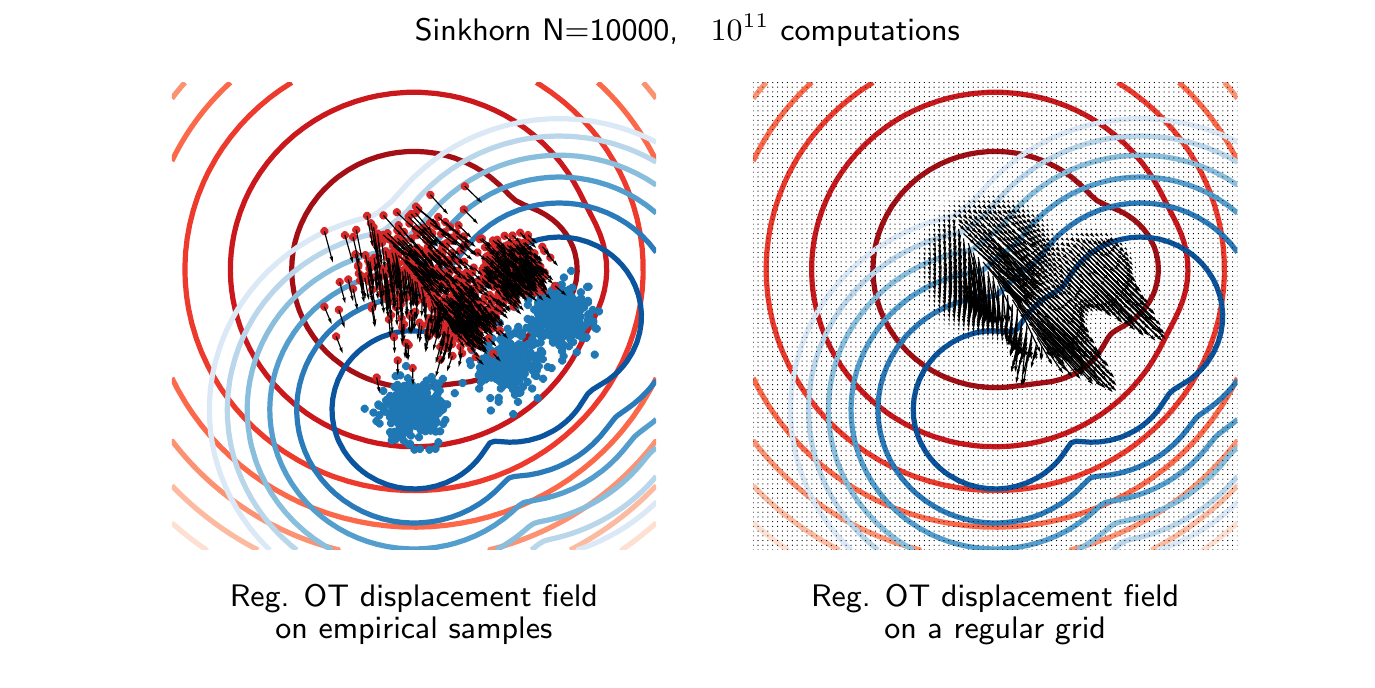}
    
    \caption{Displacement field as defined by the potentials estimated by online-Sinkhorn and Sinkhorn on a 2D GMM. With the same computational budget, online Sinkhorn finds smoother displacement fields than Sinkhorn. Those are closer to the true reference displacement field (we use Sinkhorn on $N=10000$ to estimate this reference). $\alpha$ and $\beta$ log-likelihood level-lines are displayed in red and blue, while the arrows are proportional to $\nabla_x \hat f_t(x) \d \alpha(x)$.}
    \label{fig:potentials_2d}
\end{figure}

\paragraph{Results.} As displayed in \autoref{fig:gaussian},
online Sinkhorn outperforms batch Sinkhorn for all tested batch sizes and all $\varepsilon$. It is faster and does not converge towards biased potentials. This suggests that the performance of online Sinkhorn may be underestimated in the previous analyses due to poor potential reference.

\subsection{Illustration of online Sinkhorn potentials on a 2D GMM}

The estimate $\hat f_t$ is useful to compute the gradient of the Sinkhorn
distance $\Ww(\alpha, \beta)$ with respect to the distribution $\alpha$. This is
useful when $\alpha$ is a parametric distribution $\alpha_\theta$, as it allows
to compute the gradient of the Sinkhorn distance with respect to $\theta$ using
backpropagation. For simplicity, let us assume that $\alpha = \frac{1}{n} \sum_{i=1}^n \delta_{x_i}$. Then, for all $i \in [1, n]$,
\begin{equation}
    \frac{\partial \Ww(\alpha, \beta)}{\partial x_i} = \nabla_x \big(x \to f^\star(\alpha, \beta)\big) (x_i),
\end{equation}
so that $\nabla_x f^\star(\alpha, \beta)$ provides a \textit{displacement field} that can be
descended to minimize $\alpha \to \Ww(\alpha, \beta)$. Such point of view can be
extended to general distributions using the mean-field point of view, see e.g.
\cite{santambrogio2015optimal,chizat2019sparse}.
Estimating $\nabla_x f^\star(\alpha, \beta)$ is therefore crucial to train e.g.
generator networks. Both the online Sinkhorn and the batch Sinkhorn algorithm
allow to estimate this vector field, through the plug-in estimator $x \to
\nabla_x \hat f_t$, easily computed using the form \eqref{eq:param} of $\hat
f_t$. 

\paragraph{Experiment.} With 2D GMMs, we estimate a reference vector field
$\nabla f^\star_0$ using Sinkhorn on $N=10,000$ samples and qualitatively
compare the estimations provided by online Sinkhorn and batch Sinkhorn
$(N=1,000)$, for the same number of computations.

\paragraph{Results.}We represent the estimations $\nabla_x \hat f_t$ in
\autoref{fig:potentials_2d}, for $10^8$ computations. We compare them to a
reference displacement field, estimated wityh $10^10$ computations. We observe
that online Sinkhorn estimates a smoother displacement field than batch Sinkhorn
for the same computational budget, that is closer to the reference displacement
field. In particular, it is less noisy in low-mass areas. This suggest that
online Sinkhorn would be a interesting replacement for batch Sinkhorn in
training generative architectures (used by e.g. \citet{2018-Genevay-aistats}).
$\alpha_\theta$ is then defined as the push-forward of some simple
measure with a neural network $g_\theta$. We leave this direction for future
work.

\subsection{Online Sinkhorn as a warmup process}\label{app:warmup_exp}

\begin{figure}[t]
    \centering
    $\varepsilon = 0.1$

    \includegraphics[width=\linewidth]{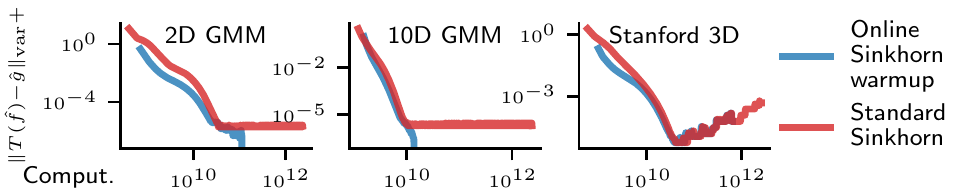}
    
    $\varepsilon = 0.01$

    \includegraphics[width=\linewidth]{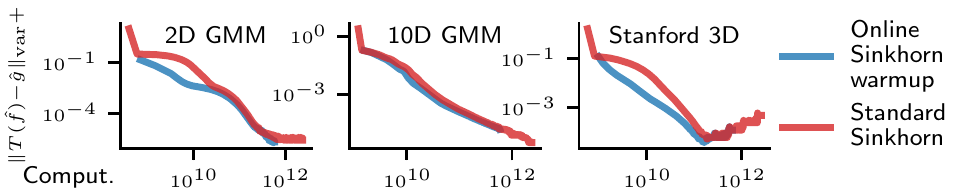}

    $\varepsilon = 0.001$

    \includegraphics[width=\linewidth]{online+full_0_001_err.pdf}

    $\varepsilon = 0.0001$

    \includegraphics[width=\linewidth]{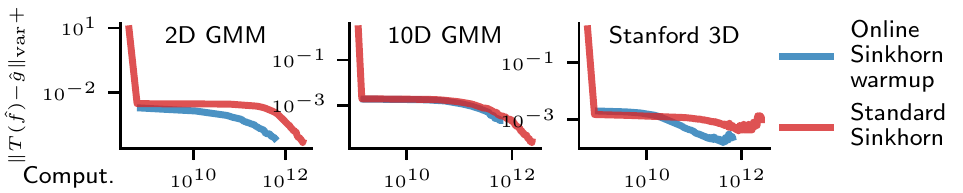}
    \caption{Performance of online-Sinkhorn as warmup for various $\varepsilon$.}
    \label{fig:warmup_full}
\end{figure}

\paragraph{Grids and details for \autoref{sec:accelerating}.} 
We set $(\eta_t, n(t)) = \big(\frac{1}{(1 + 0.1t)^a}, 100 (1 + 0.1t)^{b}\big)$,
with $(a, b) = (0, 2)$, $(a, b) = (\frac{1}{2}, 1)$ and $(a,b =1, 0)$ (constant
batch-sizes). The batch Sinkhorn algorithm that is used for reference and after
warmup uses $N=10000$. In the reference algorithm, we precompute the distance
matrix to save computation. In the warmup algorithm, this distance matrix is
filled progressively and then kept in memory to perform $C$-transforms.

We evaluated OS and fully-corrective OS, and found that fully-corrective was
less efficient (due to its higher cost in the early iterations). We evaluated
sampling with and without replacement in the warmup phase, and found sampling
without replacement to be more efficient.

\paragraph{All warmup convergence curves.} To complete \autoref{fig:warmup}, we
report convergence curves for different $\varepsilon$ in
\autoref{fig:warmup_full}. We find that speed-up increased with $\varepsilon$ and both the 2D and 3D problems, but remains limited for the 10D problem.

\vfill
\pagebreak

\section{Stochastic mirror descent interpretation}
\label{sec-mirror}

The online Sinkhorn can be understood as a stochastic mirror descent algorithm for a non-convex problem.
This equivalence is obtained by applying a change
of variable in \eqref{eq:wass}, defining 
\begin{equation}\label{eq-change-var}
	\mu \triangleq \alpha \exp(f)
	\qandq 
	\nu \triangleq \beta \exp(g). 
\end{equation}
The dual problem~\eqref{eq:dual} 
rewrites as a minimisation problem over positive measures on $\Xx$ and $\Yy$:
\begin{equation}\label{eq:wass_reparam}
    - \!\!\!\!\min_{(\mu,\nu) \in \Mm^+(\Xx)^2} \!\!\!\KL(\alpha | \mu)
    + \KL(\beta | \nu) + \dotp{\mu \otimes \nu}{e^{-C}} - 1,
\end{equation}
where the function $\KL: \Pp(\Xx) \times \Mm^+(\Xx) \triangleq \dotp{\alpha}{ \log \frac{\d \alpha}{\d \mu}}$ is the Kullback-Leibler divergence between
$\alpha$ and $\mu$. 
This objective is block convex in $\mu$, $\nu$, but not jointly convex. 
As we now detail, this problem can be solved using a stochastic mirror descent~\citep{beck2003mirror}, applied here over the Banach space of Radon measures on $\Xx$, equipped with the total variation norm. 

\paragraph{Mirror maps and gradient.}

For this, we define the (convex) distance generating function $\Mm^+(\Xx)^2 \to \RR$:
\begin{equation}
    \omega(\mu, \nu) \triangleq \KL(\alpha | \mu) + \KL(\beta | \nu).
\end{equation}
The gradient of this function and of its Fenchel conjugate $\omega^\star:
\Cc(\Xx)^2 \to \RR$ yields two \textit{mirror maps}. For all $(\mu, \nu) \in
\Mm^+(\Xx)^2$, $(\varrho, \varphi) \in \Cc(\Xx)^2, \varrho < 0, \varphi < 0$,
\begin{equation}
    \nabla \omega(\mu, \nu) = (- \frac{\d \alpha}{\d \mu}, - \frac{\d \beta}{\d \nu} )
    \qquad \nabla \omega^\star(\varrho, \varphi)
     = (-\frac{\alpha}{\varrho}, -\frac{\beta}{\varphi}).
\end{equation}
The gradient $\nabla F(\mu, \nu)$ of the objective $F$ appearing
in~\eqref{eq:wass_reparam} is a continuous function
\begin{equation}
    \nabla_\mu F(\mu, \nu) = - \frac{1}{\frac{\d\mu}{\d\alpha}} + \int_{y \in \Xx}
    \frac{\d \nu}{\d \beta}(y) \exp(- C(\cdot, y)) \d \beta(y)
\end{equation}
and similarly for $\nabla_\nu F$.

\paragraph{Stochastic mirror descent.}

To define stochastic mirror descent iterations, we may replace integration over $\beta$ is by an integration over a sampled measure $\hat \beta$. This in turn defines an \textit{unbiased gradient estimate} $\tilde \nabla F$ of $\nabla F$, which has bounded second order moments.
This absence of bias is crucial to prove convergence of SMD with high
probability. Using the mirror maps and the stochastic estimation of the
gradient, one has the following equivalence result, whose proofs stems from
direct computations. 

\begin{proposition}
The stochastic mirror descent iterations
\begin{equation}
    (\mu_t, \nu_t) = \nabla \omega^\star\Big( \nabla \omega(\mu_t, \nu_t) - 
    \eta_t \tilde \nabla F(\mu_t, \nu_t)\Big)
\end{equation}
are equal to the updates~\eqref{eq:updates} under the change of variable~\eqref{eq-change-var}.
\end{proposition}

\paragraph{Interpretation.} 

It is important to realize that $\mu_t$ and $\nu_t$ do not need to be stored in memory. Instead,
their associated potentials $f_t$ and $g_t$ are parametrized as
\eqref{eq:param}. In particular, $\mu_t$ and $\nu_t$ remain absolutely
continuous with respect to $\alpha$ and $\beta$ respectively, so that the
Kullbach-Leibler divergence terms are always finite. Note that the mirror descent
we consider operates in an infinite-dimensional space, as in \citet{hsieh2018finding}.

Finally, we mention that  when computing exact gradients (in the absence of noise) and when using constant step-size of
$\eta_t=1$, the algorithm matches exactly Sinkhorn iterations with simultaneous updates of the dual variables. This provides a novel interpretation on the Sinkhorn algorithm, that differs from the usual Bregman projection \citep{benamou2015iterative}, and the related understanding of Sinkhorn as a constant step-size mirror descent on the primal objective~\citep{mishchenko2019sinkhorn} and on a semi-dual formulation~\citep{leger2019sinkhorn}. 

Note that one can not directly apply the proofs of convergence of mirror descent to our problem, as the lack of convexity of problem \eqref{eq:wass_reparam} prevents their use.

\end{document}